\documentclass[12pt]{amsart}
\usepackage{amsfonts,amsmath,amssymb,latexsym}
\usepackage{eucal}
\usepackage{color}
\usepackage{geometry}
\usepackage{amsfonts, amsmath, amssymb,amscd,indentfirst, graphicx}
\usepackage{enumerate}
\usepackage{comment}
\usepackage{hyperref}

\newtheorem{theorem}{Theorem}[section]
\newtheorem{proposition}[theorem]{Proposition}

\newtheorem{corollary}[theorem]{Corollary}
\newtheorem{lemma}[theorem]{Lemma}
\newtheorem{definition}[theorem]{Definition}
\newtheorem{example}[theorem]{Example}
\newtheorem{remark}[theorem]{Remark}
\newtheorem*{theorem1*}{Theorem I}
\newtheorem*{theorem2*}{Theorem II}
\newtheorem*{corollary1*}{Corollary I}
\usepackage{calrsfs}
\DeclareMathAlphabet{\pazocal}{OMS}{zplm}{m}{n}

\newcommand{\R}{\mathbb{R}}

\newcommand{\n}{\mathbb{N}}
\newcommand{\LL}{\mathbb{L}}

\newcommand{\Ric}{\operatorname{Ric}}
\newcommand{\Sec}{\operatorname{Sec}}
\newcommand{\cut}{\operatorname{cut}}

\newcommand{\mt}{\overline{M}}
\newcommand{\gm}{g_M}
\newcommand{\gt}{\overline g}
\newcommand{\gte}{\overline g_\epsilon}

\newcommand{\gE}{\gt_E}

\newcommand{\nablam}{\nabla^M}
\newcommand{\nablat}{\overline \nabla}

\newcommand{\di}{\hbox{div}}
\newcommand{\divm}{\di^M}
\newcommand{\divt}{\overline{\di}}
\newcommand{\rt}{\overline{R}}

\newcommand{\ricct}{\overline{\operatorname{Ric}}}
\newcommand{\fb}{\mathbf{F}}

\begin{document}

\title[Complete translating solitons in Lorentzian products]{Complete translating solitons \\ in Lorentzian products}

\author[L. Ferrer]{Leonor Ferrer}
\address[Ferrer]{
  Departamento de Geometr\'\i{}a y Topolog\'\i{}a,
  Universidad de Granada,
  18071 Granada, Spain.
}
\email[Ferrer]{lferrer@ugr.es}
\author[F. Mart\'in]{Francisco Mart\'\i{}n}
\address[Mart\'in]{
  Departamento de Geometr\'\i{}a y Topolog\'\i{}a,
  Universidad de Granada,
  18071 Granada, Spain.
}
\email[Martin]{fmartin@ugr.es}

\author[M. S\'anchez]{Miguel S\'anchez}
\address[S\'anchez]{
  Departamento de Geometr\'\i{}a y Topolog\'\i{}a,
  Universidad de Granada,
  18071 Granada, Spain.
}
\email[S\'anchez]{sanchezm@ugr.es}

\thanks{The authors acknowledge Professors Márcio Battista (U. Alagoas) and  Marco Rigoli (U.~Milano) their explanations on some references.  Supported by the
framework of IMAG-María de Maeztu grant CEX2020-001105-M funded by MCIN/AEI/
10.13039/50110001103. F. Mart\'in  and M. S\'anchez are also partially supported by the
project PID2020-116126GB-I00 funded by MCIN/ AEI /10.13039/501100011033. This material is based upon work supported by the National Science Foundation Grant No. DMS-1928930, while F. Martín was in residence at the Simons Laufer Mathematical Sciences Institute (formerly MSRI) in Berkeley, CA, during the Fall 2024 semester.}
\date{\today}

\begin{abstract} 
Obstructions to the existence of spacelike solitons depending on the growth of the mean curvature $H$ are proved for Lorentzian products $(M\times \R, \bar g=g_M-dt^2)$ with lowerly bounded curvature. The role of these bounds for both the completeness of the soliton $\Sigma$ and the applicability of the Omori-Yau principle for the drift Laplacian, is underlined. The differences between bounds in terms of  the intrinsic distance $g$ of the soliton and the distance $g_M$ in the ambiance are analyzed, and lead to a revision of classic results on completeness for spacelike submanifolds. 
In particular, {\em primary bounds}, including affine $g$-bounds and logarithmic $g_M$-bounds,  become enough to ensure both completeness and Omori-Yau's. Therefore, they become an obstruction to the existence of solitons when the ambiance Ricci is non-negative. 

These results, illustrated with a detailed example,   deepen the  uniqueness of solutions of elliptic equations which are not uniformly elliptic, providing insights into the interplay between mean curvature growth and the global properties of this geometric flow.
\end{abstract}

\maketitle

\section{Introduction}
Translating solitons of the mean curvature flow represent fascinating geometric structures that play a crucial role in understanding the evolution of hypersurfaces in Riemannian and Lorentzian manifolds. When considering Lorentzian manifolds that are products of a Riemannian manifold and the real line, the study of translating solitons takes on an added layer of complexity and significance. These solitons serve as stationary solutions to the mean curvature flow equation, capturing the intricate balance between curvature and translation. 
 \begin{figure}[htbp]
\begin{center}
\includegraphics[width=.5\textwidth]{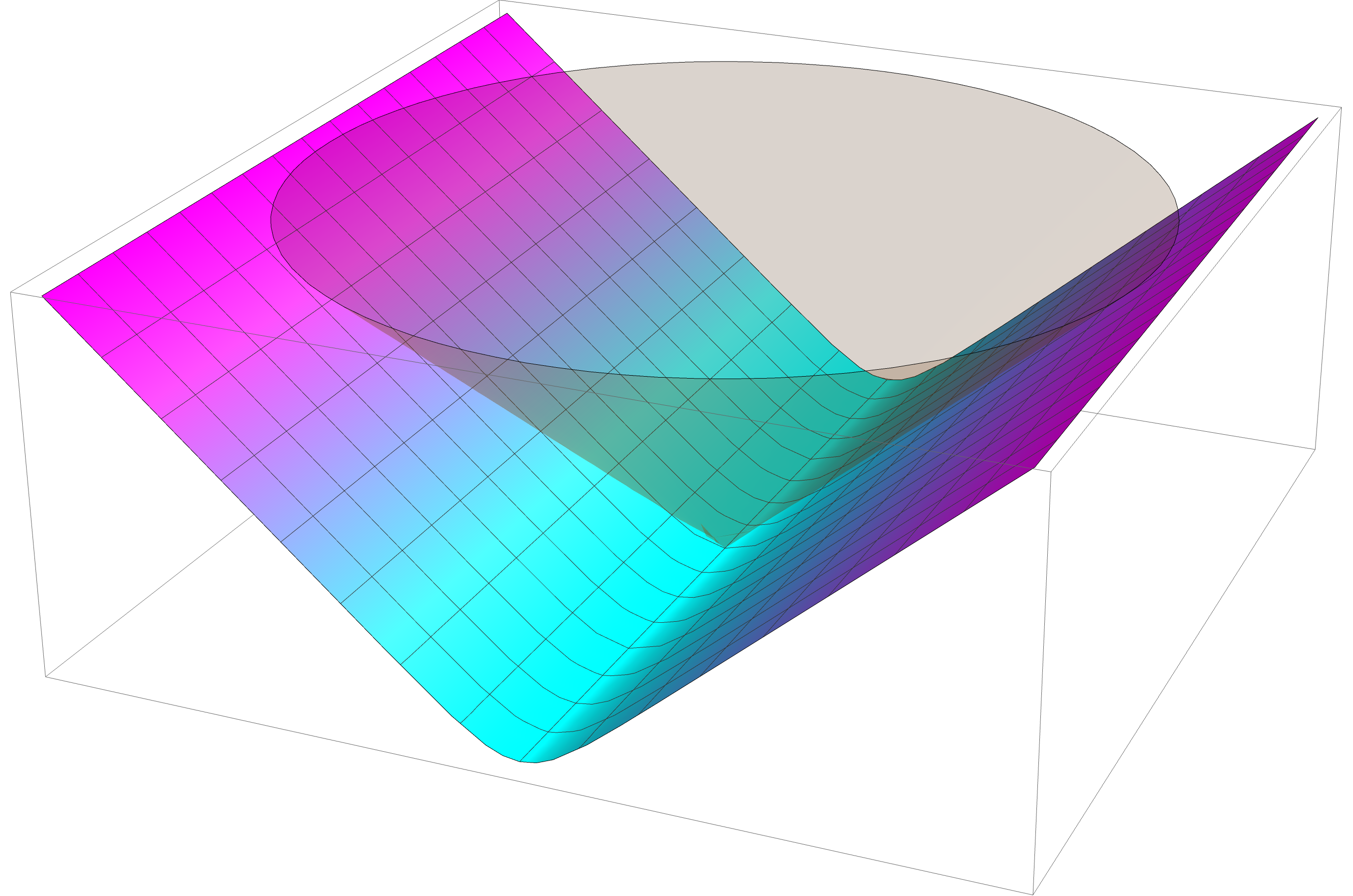}\includegraphics[width=.5\textwidth]{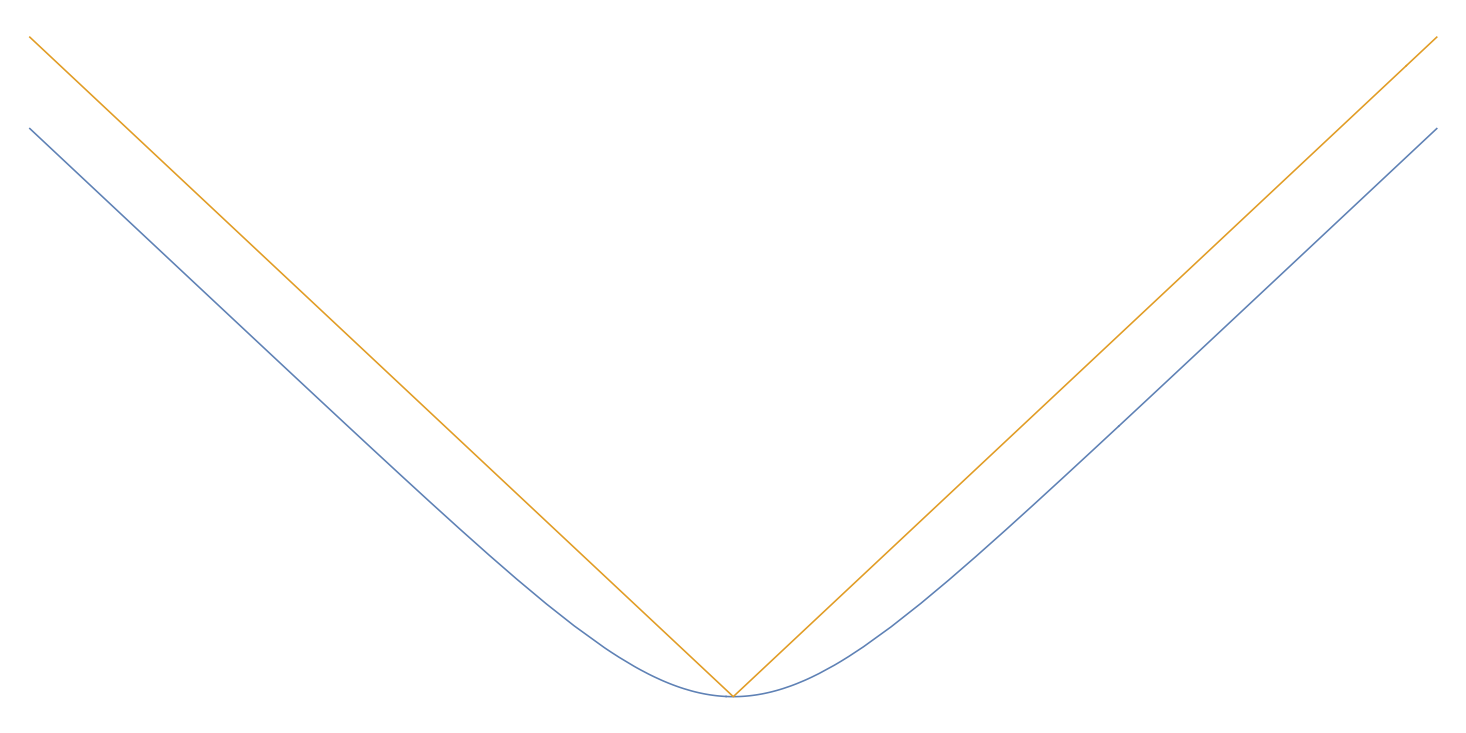}
\caption{\small Left: The grim reaper cylinder in $\LL^3$. Right: Its directrix is a spacelike curve which approaches so fast to a lightlike line that its length remains finite.  }
\label{fig:non}
\end{center}
\end{figure}
 Translating solitons
have been extensively studied. In the case of Riemannian products, background references are the thorough
recent work by Lira and Martin \cite{LM} as well as the specific case of graphical translators in the
Euclidean setting $\R^{n+1}$ by Hoffman et al. \cite{HIMW-1, HIMW-2}. In the Lorentzian setting, it
is worth pointing out the recent classification of rotational invariant translators in Lorentz-Minkowski spacetime $\LL^{n+1}$ carried out by Lawn and Ortega \cite{LO}.

A subtle issue in the context of  spacelike submanifolds in Lorentzian manifolds is  its
possible metric incompleteness. In the case of 
embedded  Riemannian submanifolds, if the image is closed as a subset then   completeness
is ensured by the completeness of the ambient manifold. However, the case of embeddings in Lorentzian manifolds becomes more delicate. For example, consider the
Riemannian Grim Reaper cylinder $x_{n+1}=\ln  (\cos x_1).$
This hypersurface is a translating soliton and it is trivially complete. 
On the contrary, the analogous Lorentzian  Grim Reaper cylinder, which is the entire graph given by $x_{n+1}=\ln  \cosh x_1$, becomes incomplete, as it is easy to check that its directrix  
$$\alpha(s) =
(s, 0, \ln (\cosh(s))
$$
is a divergent curve with finite length (see Figure \ref{fig:non}). 
This type of subtlety is not new in submanifold theory. Indeed,  an outstanding precedent is the celebrated result by Cheng and Yau \cite{CY}: {\em any  spacelike   hypersurface of constant mean curvature   $H$  closed as a subset in $\LL^{n+1}$ (then, necessarily, a graph) is complete}. In the case $H=0$, this led to solve the Bernstein problem, i.e.,  spacelike hyperplanes in $\LL^{n+1}$ are the unique {\em entire} solutions of the maximal graph equation; for $H \neq 0$, Treibers \cite{Tr} constructed counterexamples. 
Harris \cite{Ha} and Beem and Ehrlich \cite{BE} deepened in results on completeness. The former prove that, in the case of a spacelike hypersurface (in the ambience   of $b$-bounded spacetimes),  an absolute bound of its principal curvatures implies completeness. The latter proved that, for spacelike submanifolds on $\LL^{n+1}$, a subaffine growth of the Euclidean norm of the unit timelike normal vectors (with respect to the usual Euclidean metric of $\R^{n+1}$) is enough to ensure completeness.   
Notice, however, that such an Euclidean bound is not directly translatable to the (intrinsic or extrinsic) geometry  of the submanifold. Indeed, Beem and Ehrlich explain that their result cannot recover previous Harris'  for spacelike  hypersurfaces of $\LL^{n+1}$, in spite of its restrictive  absolute bound, and the subaffine growth is attained only under the additional hypothesis of being umbilic \cite[Corollary 3.7]{BE} (which can be also proved by more direct ways).

In the case of spacelike solitons,  Cheng and Qiu \cite[Theorem 3]{CQ} proved that there exist no complete $m$-dimensional spacelike translating solitons in $\R^{m+p}_p$.  Ding \cite[Theorem 1.1]{Di} exhibited (necessarily incomplete) spacelike solitons in $\LL^{n+1}$, under a general class of  behaviour of $H$ at infinity; he also noticed  that $H$ had to be unbounded. 
The case $H$ bounded was deeply studied    by Batista and Lima \cite{BL}, in the more general case of Lorentzian products $M\times \R$. Here, the boundedness of $H$ implies both the completeness of properly immersed solitons and the uniformly elliptic character of the soliton equation (see \eqref{e_soliton} below), turning out into  a {\em non-existence} result.  

In the present article, we work beyond this last result relaxing the hypotheses of boundedness for $H$. This requires a careful study of two issues, namely,  geodesic completeness (see  Section \ref{s3}) and   the Omori-Yau principle (Section \ref{s4}).

About the former, we start revisiting  the mentioned Beem-Ehrlich's study  \cite{BE} of completeness for spacelike submanifolds $\Sigma^m$. Their result for $\LL^{n+1}$ is extended to the product case $\overline{M}= M\times \R, \gt=g_M-dt^2$ and the  hypotheses are    optimized, essentially replacing the subaffine growth for   directions normal to $\Sigma^m$  by a radial 
{\em primary bound} in terms of a positive  function $G$ with $\int 1/G=\infty$ (see Theorem~\ref{t_completeness}, Remark~\ref{r_BE}).  All these bounds are expressed in terms of the extrinsic ambient Riemannian metric $\gE= g_M+dt
^2$. 
Nevertheless, in  the case of graphs,  we also give a result on completeness expressed in terms exclusively of  elements of $M$, so that it can be checked  in terms of the {\em height  function $h$} defining the graph. This is achieved in  Corollary \ref{c_complgraph}, which can be rewritten as follows.

\begin{theorem1*} Let $\left(M, g_M\right)$ be a complete Riemannian manifold and $r_M(\cdot )=$ dist$_M(o_M,\cdot )$ be the distance function  from a fixed point $o_M\in M$. 

Consider the Lorentzian  product  $\mt=M \times \mathbb{R}, \gt=g_M-d t^2$ and let
  $\Sigma$ be a spacelike hypersurface in $\mt$ obtained as the the graph of a function $u$ on $M$.   
  Let $\nu$ be the (future-pointing) unit normal to $\Sigma$ and $\nu_M$ its orthogonal projection into  $TM$.

  $\Sigma$  is complete with the  Riemannian metric inherited from $\gt$ if  
    \begin{equation}\label{e_compl_M_0}
|\nu_M|(x)\leq G(r_M(x)) \qquad   \forall x\in M,        
    \end{equation}
for a function
 $G \in C^0([0,+\infty[)$ satisfying: 
 
 (a)  $G>0$, (b) $\int_0^\infty\frac{dt}{G(t)}=\infty$ and (c) $G$ is nondecreasing.
\end{theorem1*}
It is worth pointing out:
\begin{enumerate}[1)]
    \item 
The restriction to graphs is imposed only to state the bounds in terms of the metric $g_M$. It will be harmless, because  proper spacelike immersed hypersurfaces must cover $M$ 
 and, thus, their universal coverings must be  graphs (see Proposition \ref{p_cover}). 
 
 \item \label{item2} Anyway, Theorem I can be extended to  any properly immersed spacelike hypersurface by checking \eqref{e_compl_M_0} with any of the following two metrics on $\Sigma$ (instead  of $g_M$):
 
 (A) the  metric $g_E$ induced from Riemannian $\bar g_E=g_M+dt^2$; in this case, it is not necessary to impose (c) on $G$; (see aforementioned Theorem \ref{t_completeness}), or
 
 (B)  the metric  $g$ induced from Lorentzian $\bar g$; in this case, it is not necessary to impose  (b) nor (c) on $G$ (see Remark \ref{r_caca}). 

\item The bounds for $G$ improve affine functions (see Remark \ref{r_BE}) and some applications to  soliton hypersurfaces in $\mt$ are summarized in Corollary~\ref{co:completeness}.
\end{enumerate}

  These results are expressed as a primary bound in terms of  $G$. The  condition~(c) is required only when using the metric $g_M$ and it  is slightly restrictive  (indeed, it is  implicitly assumed in Beem-Ehlich's subaffine growth, see Remark \ref{r_BE} and Example~\ref{ex_compl}). The three conditions  (a), (b)  and (c)  match formally  with the standard hypotheses to ensure the Omori-Yau maximum principle for manifolds with lowerly bounded sectional curvature say, as written in the comprehensive monograph \cite[Th. 2.5, p. 89]{AMR} (see   
  Remark \ref{r_trivial} below). However, this principle will require to compute a bound as \eqref{e_compl_M_0} with the distance $r$  for the intrinsic metric $g$ of $\Sigma$, which becomes more restrictive than the bound above for completeness. 

The Omori-Yau principle is a significant tool in differential geometry and analysis, particularly when dealing with complete Riemannian manifolds. It allows us to find an {\em Omori-Yau sequence} of points where a bounded function $u$ attains its almost maximum value (including a condition on non-positiveness for the Laplacian $\Delta$), even if a genuine maximum does not exist. The principle has been widely used to prove various geometric and analytical results, especially in the context of differential equations on manifolds.
Our  analytical issue, studied in Section \ref{s4}, will be to ensure  the Omori-Yau principle for the drift Laplacian of the height function on a $c$-translator $\Sigma$
in the non-uniformly elliptic case and with no assumption on completeness of $\Sigma$. Specifically, we obtain 
the following theorem which 
improves \cite[Theorem 3.2]{BL} by relaxing the bound of $H$  into a  primary bound in terms of $r$.

\begin{theorem2*}   Let $\left(M, g_M\right)$ be a complete Riemannian manifold  with sectional curvature 
bounded Sec$_M$ from below and consider the product manifold $\mt=M \times \R$ equipped with the metric $\gt= g_M-dt^2$. Let $F: \Sigma \to \mt$ be a  spacelike $c$-translating soliton of the mean curvature flow (not necessarily graphical) {and let $r(\cdot )=$ dist$_\Sigma(o,\cdot )$ be the distance function in $\Sigma$ from a fixed point $o\in \Sigma$}. 

The Omori-Yau principle holds for the drift Laplacian $\Delta_{c \, h}$ on $(\Sigma,g)$ (where $g=F^\ast (\gt)$ and $h$ is the height function in Definition \ref{def:h}) if 
the mean curvature satisfies
\begin{equation}  \label{eq:desH-intro} H(x) \leq G(r(x)), \quad   \forall x \in \Sigma, \,   \end{equation}
for a function
 $G \in C^0([0,+\infty[)$ satisfying: 
 
 (a)  $G>0$, (b) $\int_0^\infty\frac{dt}{G(t)}=\infty$ and (c) $G$ is nondecreasing.

 \noindent {Moreover, in the case of graphs, the bound \eqref{eq:desH-intro} is fulfilled when $H(x) \leq G^M(r_M(x)),   x \in M$, being  $G^M(s)=\int_0^sdr/G(r)$ for any function $G$ satisfying (a), (b) and (c).}
\end{theorem2*}
\noindent It is worth pointing out:
\begin{enumerate}[1)]
    \item When Sec$_M\geq 0$ the hypotheses (b) and (c)
can be removed (see Remark~\ref{re:k=0}).
\item {If one choose $G$ as an affine function, then $G^M$ becomes  a logarithm (as written explicitly in Corollary I below). }
    \item  {The estimate of $G^M$ in terms of $G$ in subsection \ref{s_extra} is sharp and, thus, so are the permitted growths of $H$ in terms of the distance $r_M$ respect to the intrinsic one $r$ for $g$. }
    
{However, the bound $\eqref{eq:desH-intro}$ is sharper  than the analogous one for $G^M$, as stressed in the example of the last section.
    This is expected, as it also happened for completeness (see the item \ref{item2}) below Theorem I). However, the $g_M$ bounds are relevant because of the limitations to ensure bounds for the hypersurfaces in  Beem-Ehrlich's result explained above. }
\end{enumerate}

 In Section \ref{sec:nonexistence}, this Omori-Yau principle  is used to  derive a  non-existence results for soliton,  which highlights the rigidity of this geometric structure. 
Specifically, our principle is combined with obstructions for the case of non-negative Ricci curvature  in \cite{BL, CQ}). This yields Theorem \ref{tII} which, in the case of graphs, can be restated from the PDE viewpoint as follows.

\begin{corollary1*}
Let $(M,g_M)$ be a complete Riemannian manifold with sectional curvature bounded from below Sec$_M\geq -c$ $(c\geq 0)$ and non-negative Ricci curvature. 

Let $c\in \R\setminus\{0\}$. 
There are no  entire  solutions on $M$ of the equation
\begin{equation}
    \label{e_soliton} {\rm div}^M\left(\frac{\nabla^M u}{\sqrt{1-|\nabla^M u|^2}}\right)=\frac{c}{\sqrt{1-|\nabla^M u|^2}},\end{equation}
with $|\nabla^M u|<1$ and satisfying: 
\begin{equation}
    \label{e_soliton2} \frac{1}{\, \sqrt{1-|\nabla^M u(x)|^2}\,} \leq G(r(x)),\quad \forall x \in M, 
\end{equation}
    
where $r(\cdot)$ is the distance in $M$ to a fixed point $o_M\in M$ with the Riemannian metric $g=g_M-du^2$ and 
 $G \in C^0([0,+\infty[)$ satisfies: 
 
 (a)  $G>0$, (b) $\int_0^\infty\frac{dt}{G(t)}=\infty$ and (c) $G$ is nondecreasing.
 
\noindent{In case $c=0$, conditions (b) and (c) can be dropped.} 

{Moreover, the bound \eqref{e_soliton2} holds if
\begin{equation}
    \label{e_soliton3} \frac{1}{\, \sqrt{1-|\nabla^M u(x)|^2}\,} \leq G^M(r_M(x)),\quad \forall x \in M, 
\end{equation}
where $r_M$ is the $g_M$-distance  to $o_M$ and $G^M(s)= \ln(As+B), s\geq 0$ (or it is chosen equal to any other  function constructed from $G$ as in Theorem II).}  
\end{corollary1*}
Notice that, here,  the growth of $G$ explicitly controls the   rate of convergence of $|\nabla^M u|$ to 1 preserving both metric completeness and the Omori-Yau principle. 

Finally, the detailed example constructed in Section \ref{sec:counter-example}, where $H$ has an affine growth,  shows the necessity of curvature bounds as above and suggests the possibility of  further issues in the study of solitons.


Taking into account our revision of  techniques in different fields, we introduce the framework of spacelike submanifolds even if our main results are focused on soliton hypersurfaces.
The paper is also written in a quite self-contained style, including some proofs of  known results with a consistent notation and framework so that the role of the different hypotheses can be better checked.

\section{Preliminaries}\label{s2}

\subsection{Background and Notation} 
\label{ssf}

Let $(M^{n}, \gm)$ be a complete Riemannian
manifold and take the natural  product \begin{equation} \label{eq:beg}\mt^{n+1}:=M^n\times \R, \quad \quad \gte:= \gm + \varepsilon \; dt^2 \quad \quad   \varepsilon^2=1,  
\end{equation}
which has Riemannian or Lorentzian signature, depending on $\varepsilon= 1$ or $\varepsilon=-1$, respectively.
The  superindeces $n, n+1$ will be removed when  no possibility of confusion exist. As we will focus in the Lorentzian case,   put $\gt := g_{\varepsilon=-1}, g_E := g_{\varepsilon=1}$.

A non-zero tangent vector $v$ is called, respectively,  timelike, lightlike or spacelike, when $\gt_\varepsilon (v,v)<0, \gt_\varepsilon(v,v)=0$ or  $\gt_\varepsilon(v,v)>0$. The two first possibilities  occur only in the Lorentzian case and, then, $v$ is called future directed if $\gt(v,\partial_t)<0$, i.e. when $v$ lies in the same cone as the natural timelike  vector field  $\partial_t$ of the product.

Given a $m$-manifold $\Sigma^m$ with $m\le n$ and $\tau_0 < 0 <\tau_1$,  consider a
smooth map
\begin{equation}
\fb:(\tau_0,\tau_1)\times \Sigma^m\to \mt
\end{equation}
such that, for all $\tau\in (\tau_0,\tau_1)$,  $\fb_{\tau} := \fb(\tau,\cdot)$ is an immersion which is spacelike,  that is, with $\fb_\tau^* \gte $
positive definite on $\Sigma$ (this is restrictive only when $\varepsilon=-1$). We denote
 $F:=\fb_0$. 
We will use consistently a bar $\overline{}$ and the superscript $^M$ for elements in $\mt$ and  $M$ and none will be used for the induced elements along the submanifolds  $\fb_{\tau}(\Sigma),\,\tau\in
(\tau_0,\tau_1),$ (i.e., on $\Sigma$ by pullback). In particular, $g$ will be the metric induced on each submanifold fromn $\gt$, and the convention will be used for the gradient/Levi Civita connections $\nablat , \nablam, \nabla$ and the divergence $\divt, \divm, \di$.


\begin{itemize}
    \item Eventually, for some expressions we will also use the notation
$$
|\cdot |_E , \qquad |\cdot |
$$
respectively, for the Riemannian norm associated with $g_E=g_M+dt^2$ and for the analogus expresion of $\gt$ on spacelike vectors (eventually $|\cdot |$ becomes the norm of $g_M$ for vectors in $TM$).

\item Given a isometrically immersed spacelike submanifold $F: \Sigma^m \rightarrow \mt^{n+1}$, we define the second fundamental form associated to $F$ as 
\begin{equation}
    \label{f_sff}
 II(X,Y):= \left( \nablat_{F_*X} F_* Y\right)^\perp,\end{equation}
where $X$ and $Y$ are smooth tangent fields on $\Sigma$. 
Given  a normal vector field $\mathbf{N}$ along $\Sigma$, we define 
\begin{equation}\label{f_sff2}
    II_{\mathbf{N}}(X,Y):= \varepsilon \; \gt(II(X,Y),\mathbf{N}).
\end{equation}
 \item When $\Sigma$ is a hypersurface, $\nu$ will be a choice of  (Gauss) unit normal vector, thus $\gt(\nu,\nu)=\varepsilon$.   In this case, we can write  
\begin{equation}
    II(X,Y) = II_\nu(X,Y) \nu, \quad \hbox{where} \; II_\nu(X,Y):=\varepsilon \; \gt(II(X,Y),\nu),\end{equation}
 which is called the scalar second fundamental form. In case $\varepsilon=-1$, $\nu$  will be chosen pointing to the  future. The Weingarten endomorphism will be denoted $A_\nu$, or simply $A$ if  there is no room to confusion. Consistently,
 \begin{eqnarray} II_\nu(X,Y) &= &\varepsilon g(A_\nu (X),Y), \label{eq:A1}\\
 - \nablat_X \nu & = & A_\nu (X). \label{eq:A2}\end{eqnarray}
 
\end{itemize}

 
\begin{definition}[Mean Curvature Flow] \label{d_mcf} 
 The submanifolds $\fb_{\tau}(\Sigma),\,\tau\in
(\tau_0,\tau_1),$ are evolving by their mean curvature vector
field if
\begin{equation}\label{eq:3}
\left(\frac{d\fb}{d\tau} \right)^\perp=\left(\fb_* \frac{\partial}{\partial \tau}\right)^\perp = \vec{H},
\end{equation}
 where
\begin{equation}
\vec{H} =\bigg(\sum_{i=1}^{m} 
\overline\nabla_{\fb_{\tau*}{\sf e}_i} \fb_{\tau*} {\sf e}_i\bigg)^\perp
\end{equation}
is the (non-normalized) mean curvature vector of $\fb_\tau$. Here and in what follows $\perp$ indicates the projection onto the normal bundle; the local  tangent frame $\{{\sf
e}_i\}_{i=1}^m$ is orthonormal with respect to the metric induced in $\Sigma$ by $\fb_\tau$. 
\end{definition}

The natural  vector field $\partial_t$ is  parallel  on $\mt$. We set $\Phi:\mathbb{R}\times \mt \to \mt$ to denote the
flow generated by $\partial_t$, namely, 
\[
\Phi(s, (x,t)) = \Phi_s(x,t)=(x,s+t), \quad (x,t)\in \mt^{n+1}, \,\, s\in \mathbb{R}.
\]

Motivated by the above geometric setting, we define a notion of translating soliton  with respect to the vector field $\partial_t$ as follows: 

\begin{definition}[Translating soliton]\label{soliton-definition}
An isometric immersion $F:\Sigma^m\to \mt^{n+1}_ \varepsilon$ is a \emph{mean
curvature flow soliton}  with respect to $\partial_t$ of velocity $c$ (or a $c$-translator for short) if
\begin{equation}
\label{solitonA-2} c \, \partial_t^\perp=\vec{H}
\end{equation}
along $F$ for some constant $c\in \mathbb{R}$, which will be assumed $c>0$ with no loss generality\footnote{If necessary we can compose with the isometry of $\mt$ given by $(x,t) \mapsto(x,-t).$}. With a slight abuse of notation, we also say that
the submanifold  $F(\Sigma^m)$ itself is the translating
soliton {\rm (with respect to the vector field $\partial_t$}. If $m=n$, that is, for codimension $1$, we  just write  $\Sigma$ and the condition
becomes
\begin{equation}
\label{soliton-scalar} H=\varepsilon c\,\gt ( \partial_t, \nu),
\end{equation}
where the mean curvature $H$ with respect to the normal vector field $\nu$ along $F$ is given by
\begin{equation}
\label{HHN}
\vec{H} = H \cdot \nu.
\end{equation}
\end{definition}
\begin{remark}[Sign of the mean curvature]\label{r_signH}
    It is important to notice that, from our conventions, one has:
    \begin{equation} \label{eq:traceA}
        \operatorname{trace}(A)=\varepsilon H \qquad
        \hbox{so that, for translating solitons with $\epsilon=-1$,} \quad H>0 .
    \end{equation}
\end{remark}
\begin{remark}\label{r_1}
    Notice that $F:\Sigma^m\to \mt^{n+1}$ is a mean curvature flow soliton with respect to $\partial_t$
     if $$ \fb(x,s):= \left(\Phi_{c\,s} \circ F \right)(x)$$ is a mean curvature flow in $\mt^{n+1}$ in the sense of \eqref{eq:3}.   In what follows, we will consider this flow for solitons, thus setting also the tangent part in Definition \ref{d_mcf}.
\end{remark}

\begin{definition} \label{def:h}
    For any given translating soliton $F: \Sigma \to \mt$, we  consider the height function $h : = t \circ F$, i.e., the restriction of the $t$-function to the submanifold.
\end{definition} Then, it is straightforward to check (from the choices above 
that
\begin{equation}
    \label{eq:funcion-h}
    F_*\nabla h = \varepsilon \, \partial_t^\top ,
\qquad  |\nabla h|^2 =\epsilon \left(1-\frac{H^2}{c^2}\right)  \end{equation}
where $^\top$ denotes component tangent to the hypersurface.
\begin{example}\label{product-example} {\rm (Translating graphs.)}  
Consider a domain $\Omega \subseteq M$ and a smooth function $u: \Omega  \to \R$ such that $1+\varepsilon |\nabla^M u|^2>0$ (this will  be a restriction only when $\varepsilon=-1$).  Let us define $\Sigma:={\rm Graph}(u)= \{ (x,u(x)) \; : \; x \in \Omega\}.$
It is not hard to see that $\Sigma$ is a translating soliton if, and only if, $u$ satisfies the quasilinear elliptic equation
\begin{equation}
\label{pde-parabolic}
{\rm div}_M \Big(\frac{\nabla^M u}{W}\Big)= \frac{c}{W},
\end{equation}
where 
\[
W = \sqrt{1+ \varepsilon |\nabla^M u|^2},
\]
 This notion of translating soliton has been extensively studied in $\R^{n+1}$ and $\LL^{n+1}$ (see for instance \cite{HIMW-1, LO}) 
 and this setting is the natural extension to (semi)-Riemannian products $M \times \mathbb{R}$. 
\end{example}

\subsection{Relevant properties of the height function $h$} 
The analysis of the height function (see Definition \ref{def:h}) for a translator  plays a fundamental role in this paper, and its properties are summarized in the following lemma. Its first assertion corresponds to equation (2.4) in \cite{BL}. Its proof is simple and it is included for self-containedness.  
The second one can be regarded as a particular case of a first variation formula for $H$ with respect to a deformation of the surrounding space generated by a transverse vector field. This assertion was also proved in \cite[Lemma 4.1]{BL} (see also \cite{AC}). Again, we have included the proof in order to make our paper self-contained (taking also into account that our conventions about the mean curvature are slightly different from the ones in \cite{BL}). More details about this variation formulas can be found in \cite{B1, CB, EH, FM}.

\begin{remark}\label{r_Codazzi} Along the paper, we use the convention $R(V,W)=\nabla_{[V,W]}-[\nabla_V,\nabla_W]$  so that, for $V,W,Z$ tangent to the hypersurface, Codazzi equation reads: $$\gt( (\nabla_V A)W,X)-\gt( (\nabla_W A)V, X) = 
\rt(V,W,\nu,X).$$  \end{remark}

\begin{lemma}\label{lem:alias}  Let $\left(M, \gm\right)$ be Riemannian and   $(\mt^{n+1}, \gt=g_M+\varepsilon d \, t^2)$ as above.  Let $F: \Sigma^n \rightarrow \mt$ be a spacelike $c$-translating soliton (not necessarily graphical). Then, given  a tangent vector field $X$  along $\Sigma$, we have that
\begin{equation}
    \label{lem:h}
\nabla^2 (c \,h)(X,X)=\varepsilon H \; g(A X,X),
\end{equation}
where $A$ is the Weingarten endomorphism   in   \eqref{eq:A2} under conventions \eqref{eq:A1}, \eqref{eq:traceA}. 

Moreover,
\begin{equation}
     \label{f_deltaH}
-\varepsilon \Delta H= g(\nabla H, \nabla(c h))+\left(\operatorname{Ric}_M\left(\nu_M, \nu_M\right)+|A|^2\right) H ,
\end{equation}
where $\nu_M$ represents the orthogonal projection of the Gauss normal 
$\nu$ onto $M$. 
\end{lemma}

\begin{proof}
    From the definition of the Hessian, \eqref{eq:funcion-h} and  that $\partial_t$ is parallel in $\mt$:
    \begin{multline*}
        \nabla^2 h(X,X)=  g( \nabla_X \nabla h, X)=  \gt( \nablat_X\left[\varepsilon \partial_t-\gt (\partial_t,\nu) \nu\right], X)\\ = - \gt( \partial_t ,\nu) \, \gt(\nablat_X\nu,X)=\gt( \partial_t ,\nu) \,\gt(AX,X).
    \end{multline*}
As being a $c$- translating soliton means $H=\varepsilon \, c \, \gt (\partial_t,\nu)$,
this expression becomes
$$  \varepsilon \frac{H}{c} \gt(AX,X),$$
as required.

For the last assertion, as $\Sigma$ is a translator, 
for any $X$ tangent to $\Sigma$: 
$$
\begin{aligned}
\gt(X, \nabla H) & =X(H)=X\left(\varepsilon c \gt\left(\nu, \partial_t\right)\right)=\varepsilon c \gt\left(\nablat_X \nu, \partial_t\right) \\
& =-\varepsilon c \gt\left(A X, \partial_t\right)=-\varepsilon c \gt\left(A X, \partial_t^{\top}\right)=-\varepsilon c \gt\left(X, A \partial_t^{\top}\right), 
\end{aligned}
$$
that is,
\begin{equation}\label{e_H_A}
    \nabla H=-\varepsilon c A \partial_t^{\top}.
\end{equation}
Using again $ \nablat_X  c \partial_t=0$, and (from \eqref{solitonA-2}, \eqref{HHN}) $c \partial_t=c \partial_t^{\top}+H \nu$, 
\begin{equation} \label{eq:12}
\begin{aligned}
0 & =\nablat_X (c \partial_t^{\top})+\nablat_X (H \nu)=\nabla_X (c \partial_t^{\top})+\Pi\left(X, c \partial_t^{\top}\right)  +X(H) \nu-H \, A X,
\end{aligned}
\end{equation}
and focusing on the tangent components, 
\begin{equation}\label{eq:12bis}
    \nabla_X (c \partial_t^{\top})=H A X.
    \end{equation}
Now, consider the following expression from \eqref{e_H_A}:
\begin{multline} \label{eq:12.5}
g\left(\nabla_X \nabla H, X\right)  = g\left(\nabla_X\left(-\varepsilon c A \partial_t^{\top}\right), X\right) \\ = -\varepsilon c\left(g\left(\left(\nabla_X A\right) \partial_t^{\top}, X\right)+g\left(A\left(\nabla_X \partial_t^{\top}\right), X\right)\right) 
  \\ = -\varepsilon c \left(g\left(\left(\nabla_X A\right) \partial_t^{\top}, X\right) + g\left(\nabla_X \partial_t^{\top}, A X\right)\right) \\ = -\varepsilon c g\left(\left(\nabla_{\partial_t^{\top}} 
 A\right) X, X\right) 
 +\varepsilon c \gt\left(\rt\left(X, \partial_t^{\top}\right) X, \nu\right)-\varepsilon  H g(A X, A X)
\end{multline}
For the last identity, use Codazzi's  (as in Remark \ref{r_Codazzi})
 and $g\left(\nabla_X (c \partial_t^{\top}), A X\right)=H g(A X, A X)$ from \eqref{eq:12bis}.

Next, this formula will be contracted and evaluated  at a fixed $p \in \Sigma$ by using an orthonormal frame $\left\{E_i\right\}_{i=1}^n$  which is parallel and orthonormal at $p$ and $\left\{E_i(p)\right\}_{i=1}^n$ is a base of eigenvectors of $A$, that is, $A E_i(p)=k_i(p) E_i(p)$. 
Let us focus on the term:

\begin{multline} \label{eq:13}
\sum_{i=1}^n \varepsilon c g_p\left(\left(\nabla_{\left.\partial_t^{\top}\right|_p} A\right) E_i, E_i(p)\right)   
\\ = \sum_{i=1}^n \varepsilon c \left(g_p\left(\nabla_{\left.\partial_t^{\top}\right|_p} A E_i, E_i(p)\right)-g_p\left(A \nabla_{\left.\partial_t^{\top}\right|_p} E_i, E_i(p)\right)\right) \\
 =\sum_{i=1}^n \varepsilon c\left(\left.\partial_t^{\top}\right|_p g\left(A E_i, E_i\right)-2 g_p\left(\nabla_{\left.\partial_t^{\top}\right|_p} E_i, A E_i(p)\right)\right) 
\end{multline}
The last term vanishes at $p$ by our choice of the frame at $p$, thus  \eqref{eq:13} becomes:
$$
\begin{aligned}
\sum_{i=1}^n \varepsilon c g_p\left(\left(\nabla_{\left.\partial_t^{\top}\right|_p} A\right) E_i, E_i(p)\right) & 
=\left.\partial_t^{\top}\right|_p\left(\sum_{i=1}^n \varepsilon c g\left(A E_i, E_i\right)\right) \\
& = c \left.\partial_t^{\top}\right|_p(H)=  c g_p\left(\left.\partial_t^{\top}\right|_p, \nabla H(p)\right)
\end{aligned}
$$
Finally, contracting \eqref{eq:12.5} at each  $p$, we get on the entire hypersurface
\begin{equation}\label{eq:juras}
    \Delta H(p)=-c g_p\left(\nabla H(p),\left.\partial_t^{\top}\right|_p\right)+\varepsilon c \ricct_p\left(\left.\partial_t^{\top}\right|_p, \nu(p)\right)-\varepsilon H(p)|A|(p)^2 .
\end{equation}
 From \eqref{eq:funcion-h}, we have that $c g\left(\nabla H, \partial_t^{\top}\right)=\varepsilon g(\nabla H, \nabla(c h))$. 
 Then, equation (\ref{eq:juras}) becomes:
\begin{equation}\label{eq:juras-2}
\Delta H=-\varepsilon g(\nabla H, \nabla(c h))+\varepsilon  \ricct\left(c \partial_t^{\top}, \nu\right)-\varepsilon H|A|^2 .
\end{equation}

From \eqref{solitonA-2}, \eqref{HHN}, and from
\eqref{soliton-scalar} we have  $c \partial_t=c \partial_t^{\top}+H \nu$ and  $\nu=\nu_M+\frac{H}{c} \partial_t$, respectively. Then, plugging the latter in the former, we get
$$
c \partial_t^{\top}=\left(c-\frac{H^2}{c}\right) \partial_t-H \nu_M .
$$
Therefore, 
$$
\begin{aligned}
\varepsilon \ricct\left(c \partial_t^{\top}, \nu\right) & =\varepsilon \ricct\left(\left(c-\frac{H^2}{c}\right) \partial_t-H \nu_M, \frac{H}{c} \partial_t+\nu_M\right) \\
& =-\varepsilon H \ricct\left(\nu_M, \nu_M\right)=-\varepsilon H \operatorname{Ric}_M\left(\nu_M, \nu_M\right) .
\end{aligned}
$$
If we include the above information in \eqref{eq:juras-2}, we obtain the required formula \eqref{f_deltaH}.
\end{proof}

\section{Completeness results}\label{s3}

Geodesic completeness as well as the completeness of vector fields under subaffine growth or, more precisely, under a {\em primary bound}, is well understood in different contexts, see the review \cite[Def. 1, Th. 1]{SaFields}  or recent \cite[Sect. 3]{EFSZ}. 
The case of completeness of spacelike submanifolds in  Lorentz-Minkowski spacetime was studied by Beem and Ehrlich \cite{BE}. Here, we will extend   their results to the case of a 
product $M\times \R$, discussing and sharpening their hypotheses in order to apply them to solitons.

\subsection{Completeness of spacelike submanifolds in the Lorentzian product $M\times \R$} \label{s3_1}

Next,  the following result which extends  \cite[Theorem 3.4]{BE} is proven and discussed.

\begin{theorem}\label{t_completeness} Let $\left(M, g_M\right)$ be a complete Riemannian manifold. Consider the product manifold $\mt=M \times \mathbb{R}$ and the metrics $\gt=g_M-d t^2$ and $\gE=g_M+d t^2$. Let $F: \Sigma^m \rightarrow \mt$ be a proper immersion of $\Sigma^m$ as a spacelike submanifold of $(\mt, \gt)$. 
Let  $|\cdot|_E $ be the norm with respect to the Riemannian metric $\gE:= g+dt^2$ and  put
$$
|\nu^{min}|_E(x):=\hbox{Min}\{|\nu(x)|_E: \nu(x) \; \hbox{is timelike, unit, normal to} \;T_{F(x)}\Sigma^m\}, \quad \forall x\in \Sigma^m. 
$$
Let
$r_E(\cdot)=\operatorname{dist}_\Sigma^E(o, \cdot)$ be the distance function from a fixed point $o \in \Sigma$ respect to the Riemannian metric $g_E:=F^*\left(\gE\right)$. Assume that $|\nu^{min}|_E$ is primarily bounded, that is,
\begin{equation}\label{e_primarybound}
|\nu^{min}|_E(x) \leq G\left(r_E(F(x))\right), \quad  \forall x \in \Sigma, 
\end{equation}
where $G: [0,\infty)\rightarrow [0, \infty)$,  is a continuous function satisfying:
$$
\text { (a) } G>0 \qquad \text{ and} \qquad \text{(b) } 
\int_{0}^\infty\frac{dr}{G(r)} =\infty.
$$
 Then, the natural induced metric $g= F^* \gt$ on $\Sigma^m$ is complete. 
 \end{theorem}

\begin{remark}\label{r_complete} (1) Even though at each point $x$ a vector $\nu$ with minimum $|\nu_E|(x)$ must exist, the global existence of such a vector field is not required.

(2)    In  case that $\Sigma$ is a hypersurface, the bound \eqref{e_primarybound} applies just putting $|\nu^{min}|_E(x) = |\nu|_E(x) $ for $\nu$ in the unique normal direction.  

(3) In the case of a graph $u$, the bound \eqref{e_primarybound} under hypotheses (a) and (b) for $G$ becomes equivalent to the bound
    \begin{equation}
        \label{e_compl} |\nabla^M u|^2(x) \leq 1-\frac{1}{G(r_E(x))^2}.
    \end{equation}
    (see  Prop. \ref{p_extra} for related computations). In particular, one can choose $G$ as an affine function $G(r)=Ar+B$, with $A,B>0$ obtaining then the sufficient condition
    $$
    |\nabla^M u|^2(x) \leq 1-\frac{1}{(Ar_E(x)+B)^2}
    $$
(see   Remark \ref{r_BE} for finer estimates).

\end{remark} 
For the proof, consider first the following two lemmas; the first one a straightforward extension of  \cite[Lemma 3.1]{BE}, included  here for the convenience of the reader. 
\begin{lemma}  \label{lem:X}
Let $\nu$ be a unit timelike normal vector field in $\mt.$ Consider  a unit spacelike vector field $X$ such that $\gt(X,\nu)=0,$ then
\[ |X|_E \leq |\nu |_E,\]
 Moreover, the equality holds if, and only if, $X$, $\nu$ and $\partial_t$ lie in the same $2$-plane.
\end{lemma}
\begin{proof}
    We decompose $\nu=\nu_M+ \nu_{\R} \cdot \partial_t$ and $X= X_M+ X_{\R} \cdot \partial_t$, where the subscript $(\cdot )_M$ means the projection to $TM.$
Using that $\nu$ and $X$ are unit:
\begin{equation} \label{eq:E}
|\nu|_E^2-|X|_E^2= |\nu_M|^2 +\nu_{\R}^2-|X_M|^2-X_{\R}^2=2 \left( 1+|\nu_M|^2 -|X_M|^2\right),
\end{equation}
    and using that they are orthogonal, $\nu_{\R}X_{\R}= g(X_M,\nu_M)$, thus,
\begin{equation} \label{eq:perp}
 g(X_M,\nu_M)^2= |\nu_M|^2 \,   |X_M|^2 -|\nu_M|^2 +|X_M|^2-1.
\end{equation}
Using \eqref{eq:E} and \eqref{eq:perp} and taking Cauchy-Schwartz inequality into account, we get:
\begin{equation}
    \label{eq:ine-norm} 
    |X|_E^2-|\nu |_E^2= 2 \left( g(X_M,\nu_M)^2- |\nu_M|^2 \,   |X_M|^2\right) \leq 0.
\end{equation}
The equality only holds if and only if $X_M$ and $\nu_M$ are colinear, that is, when  Span$\{\partial_t, \nu, X\}$ = Span$\{\partial_t, \nu_M, X_M\}$is a plane.
\end{proof}

\begin{lemma}
    Let $(\mt,\gE)$ be a complete Riemannian manifold and $\gamma:[0,b) \rightarrow \mt$ a curve not contained in a compact subset. If $|\gamma'|_E$ is primarily bounded as in \eqref{e_primarybound} then $\gamma$ is complete, i.e., $b=\infty$.
\end{lemma}
    
\begin{proof}\label{l_compl}
    Consider $o$, $r_E$ as above and the  function 
    $$r:[0,b)\rightarrow [0,\infty), \qquad t\mapsto r(t):=r_E(\gamma(t)).$$
    This function $r$ is onto because of the completeness of $g_E$ and it is easy to check that it is a.e. differentiable using that the distance function is Lipschitz\footnote{Notice that this happens even if $\gamma(t)$ remains in the cut locus of $o$ for $t$ in a set of non-zero measure (see for example,  the proof of \cite[Prop. 3.1]{EFSZ}).}. What is more, $|\nabla^E r|=1$ when smooth and, denoting by a dot the a.e. derivative:
    $$
    |\dot r(t)| \leq |\gamma'(t)|_E \leq G(r(t)),
    $$
    the last inequality by hypothesis and the definition of $r$. Integrating the inequality $1\geq |\dot r|/G(r)$,
    $$
    t\geq \int_0^t \frac{|\dot r(t)|}{G(r(t))}dt \geq \int_0^t \frac{\dot r(t)}{G(r(t))}dt=\int_0^{r(t)} \frac{dr}{G(r)}
    $$
    (for the last equality, recall the sign preserving change of variables in dimension 1).
    As there are points with $r(t)$ arbitrarily big, the integral on the right diverges by the hypothesis (ii) on $G$, thus,  $b$ cannot be finite.
\end{proof}

\begin{proof} [Proof of Theorem \ref{t_completeness}.]  Let $\gamma:[0, b) \rightarrow \Sigma$ be an inextendible unit geodesic in $(\Sigma,g)$. Assuming by contradiction $b<\infty$, $\gamma$ cannot be contained in any compact subset. 

By Lemma \ref{lem:X}, $\left|\gamma^{\prime}(t)\right|_E \leq|\nu^{min}(\gamma(t))|_E$ and, by hypothesis,
\begin{equation} \label{eq:5.2.3}
 \left|\gamma^{\prime}(t)\right|_E \leq G\left(r_E(x)\right) .
  \end{equation}
So, Lemma \ref{l_compl} applies yielding the contradiction $b=\infty$. 
\end{proof}

\begin{remark}[Comparison with  Beem and Ehrlich's result] \label{r_BE}
(1)    In \cite[Th. 3.4]{BE}, $G$ is chosen directly as an affine  function $G(r)=Ar+B$, which is the optimal polynomial growth to obtain completeness. However, it is still possible to choose    better non-polynomial growths such as   $G(r)=A \, r\ln^k(1+r) +B$ with $k>0$.\footnote{Other technical  conditions explored in \cite[p. 218]{BE} are also improved. 
What is more, it is straightforward to check that our primary bound becomes not only sufficient but also necessary  for $\LL^2$ (compare with \cite[Corollary 3.6]{BE}).}
    
 (2)   For these estimates of growth,  $G'\geq 0$. This assumption is natural in order to define types of  {\em growth} and it permits to simplify proofs\footnote{Compare with the proof of \cite[Th. 3.4]{BE}. In particular, notice that, in our Lemma \ref{l_compl}, if one takes the unit speed reparametrization  $\hat \gamma$ of $\gamma$, (which  satisfies $r_E(\hat \gamma(t))\leq t$) then $G(r_E(\hat \gamma(t))\leq G(t)$.}. However, the next example shows that   this assumption imposes a  restriction. 
\end{remark}

\begin{example} \label{ex_compl}
   In Lorentz-Minkowski $\mt=\R^2, \gt=dx^2-dt^2$, construct a spacelike curve $\gamma(x)=(x,u(x)), x\in \R$ (which will be regarded as a spacelike hypersurface), by repeating the following pattern  in each interval of length five 
   $I_n:= [5n,5(n+1)]$, for all $n$ non-negative integer:
   $$\begin{array}{rll}
   u(x)= 0     &  \hbox{if} & x\in (-\infty,0] \\
   u(x)= 0     &  \hbox{if} & x\in  [5n,5n+1] \cup [5n+4,5(n+1)] \\
       u(x)= 1     &  \hbox{if} & x\in [5n+2+\frac{1}{2^{|n|+1}},5n+3-\frac{1}{2^{|n|+1}}]. 
   \end{array}
   $$
   It is not difficult to check that such a smooth $u$ exists, as the interval is chosen such that  $|\dot u:=du/dx|<1$, thus  preserving the spacelike character. However, by the mean value theorem, $\dot u$ must satisfy, at some point $x_n$ of each interval $I_n$,
   $$\dot u(x_n)=\frac{1}{1+2^{-(|n|+1)}}=\frac{2^{|n|+1}}{2^{|n|+1}+1}.$$
   Taking into account \eqref{e_compl}, the optimal function $G$ to ensure completeness is:
   $$
   G(x)=\frac{1}{\sqrt{1-\dot u^2}}, \qquad \hbox{thus} \qquad G(x_n)\geq 2^{(|n|-1)/2}. $$
   Clearly, this choice of $G$ satisfies the requeriments of Theorem \ref{t_completeness}. However, no monotonous function $G^*$ can satisfy them. Indeed, $G^*$ should satisfy $G^*(x_n)\geq G(x_n)\geq 2^{(|n|-1)/2}$, which would turn out incompatible with hypothesis (b) therein.
\end{example}

Nevertheless, if $G$ is monotonous, a condition similar to \eqref{e_primarybound}
in terms of the $M$ part will suffice for graphs (Corollary \ref{c_complgraph}). This is specially interesting, as the distance $d_M$  of $M$ is known a priori.

\subsection{The case of hypersurfaces and graphs}\label{s3_2a}
Next, we will restrict to  codimension one. 
We start with some preliminary properties of hypersurfaces.
\begin{proposition} \label{p_cover}
     Consider the Lorentzian product $\mt=M \times \mathbb{R}$, $\gt=g_M-d t^2$ and  
any   immersed spacelike hypersurface $\Sigma$:

(1) If $g_M$ is incomplete then $\Sigma$ is incomplete.

(2) If $g_M$ is complete and $\Sigma$ is proper, then the projection $\Pi: \Sigma \rightarrow M$ is a covering map. Thus, in this case $\Sigma$ is complete if and only if the graph naturally associated with the universal covering $\tilde \Sigma$ in $\tilde M\times \R$ (see  \eqref{e_universal} below) is complete.
\end{proposition}

\begin{proof} We will use that $\Pi$ is a local diffeomorphism for spacelike hypersurfaces.

(1)  If $x_0\in \Pi(\Sigma)$ and $\gamma:[0,b)\rightarrow M$, $b<\infty$, is a unit diverging curve in $M$ starting at $x_0$, then there exists an inextensible lift $\tilde \gamma: [0,b')\rightarrow \Sigma$, where $ b'\leq b$, so that $\Pi\circ \tilde\gamma$ $ =\gamma|_{[0,b')}$. As $g(\tilde \gamma'(s),\tilde \gamma'(s))\leq g_M(\gamma'(s),\gamma'(s))$, the length of $\tilde \gamma$ is bounded by $b'$ and incompleteness follows.    

    (2)   Endowing $\Sigma$ with the pullback metric $\tilde g_M:=\Pi^*g_M$,  
    $\Pi$ becomes a local isometry, and it is enough to check that it is complete  (apply for example \cite[Corollary 7.29]{Oneill}).
    
     Otherwise,  there exists  a unit inextensible curve  $\tilde \gamma: [0,b)\rightarrow \Sigma$ of finite $\tilde g_M$-length $b$. Let us prove that $\tilde \gamma$ lies in a compact subset of $\mt$, in contradiction  with properness.
    
The projection $\gamma:= \Pi\circ \tilde \gamma$ can be extended to a point 
   $x_b:=\lim_{s\rightarrow b} \gamma(s)$ by the completeness of $(M,g_M)$. The existence of convex neighborhoods in Riemannian manifolds ensures that, for small $\epsilon>0$,  there exists a point $x_\epsilon=\gamma(b-\epsilon)$ such that the closed ball $\bar B(x_\epsilon, \epsilon)$ is included in a normal neighborhood of $x_\epsilon$ in $(M,g_M)$, and also that the compact set $\bar B(x_\epsilon, \epsilon)\times [-\epsilon,\epsilon]$ is included in a normal neighborhood of $\tilde \gamma(b-\epsilon)$ in $(\mt,\gt)$. Then, this compact set includes $\tilde \gamma|_{[b-\epsilon,b)}$, as required.

   For the last assertion, 
   the universal covering $\tilde \Sigma$ of $\Sigma$ yields the commutative diagram
of local isometries (with the metric naturally induced on $\tilde \Sigma$) 
\begin{equation}
    \label{e_universal}
\begin{array}{ccc}
   \tilde \Sigma  &  & \tilde M\times \R \\
   \downarrow   & \searrow & \downarrow \\
   \Sigma & \hookrightarrow & M\times \R
\end{array} 
\end{equation}
which induces an isometric  lift $\tilde \Sigma \hookrightarrow \tilde M\times\R $ so that $\tilde M$
can be regarded as a graph on the universal covering $\tilde M$ of $M$. 
\end{proof}

Part (1) justifies to assume the completeness of $g_M$ in what follows. By (2), one can focus on the case of graphs, otherwise taking lifts in the universal covering. 

\begin{corollary}\label{c_complgraph} Under the hypotheses of Theorem \ref{t_completeness} except  \eqref{e_primarybound}, assume that $\Sigma$ ($= \Sigma^{m=n}$) is the graph of a function $u$.
Assume additionally: $(c)$  $G$ is non-decreasing.

Let $r_M(\cdot )=$ dist$_M(o_M,\cdot )$ be the distance function in $(M,g_M)$ from a fixed point $o_M\in M$ and $\nu_M$ the orthogonal projection of $\nu$ on  $TM$. 
 If the inequality  
    \begin{equation}\label{e_compl_M}
|\nu_M|(x)\leq G(r_M(x)) \qquad   \forall x\in M,        
    \end{equation}
holds (instead of \eqref{e_primarybound}), then the hypersurface $(\Sigma, g)$ is complete. 
\end{corollary}
\begin{proof} Notice that it is sufficient to prove that 
the hypothesis \eqref{e_primarybound} (or the equivalent one for graphs \eqref{e_compl}) holds. First, let us check that \eqref{e_primarybound} holds if 
\begin{equation}
\label{e_aux}
 |\nu|_E(x) \leq G(r_M(x)) \qquad   \forall x\in M.
\end{equation}
    As $g_E$ applied on vectors tangent to $\Sigma$ is non-smaller than $g_M$ applied on their projections in $M$, one has
    $r_M(x) \leq r_E(x)$ for all $x\in M$.
    As $G$ is non-decreasing now, if \eqref{e_aux} holds then 
$$|\nu|_E(x) \leq G(r_M(x))\leq G(r_E(x)), 
$$
   and \eqref{e_compl} holds.  
   Finally, notice that
   $$-1=\gt(\nu,\nu)=-\gt(\nu,\partial_t) +g_M(\nu_M,\nu_M), \qquad \hbox{thus} \qquad  
    |\nu|_E^2\leq 2 |\nu|^2.$$
   So, the bound \eqref{e_compl_M} implies  \eqref{e_aux} up to  changing $G$ by $\sqrt{2}G$, which is irrelevant because the latter satisfies the same properties required for the former. 
\end{proof}    

\begin{remark}\label{r_caca} Corollary \ref{c_complgraph} proves Theorem I in the Introduction, and the additional results below this theorem follow from the following considerations.
\begin{enumerate}    \item 
The distance $d$ on $\Sigma$ associated with $g$ satisfies  $d\leq d_M \leq d_E$.

Both $d$ and $d_E$ makes sense directly even if the soliton is not a graph. 

\item In  formula \eqref{e_aux}, one could use also $g$, its norm and  $r$  instead of $g_M$ its norm and $r_M$  in order to prove the completeness of $d$

However, such a substitution by a smaller metric yields a weaker result. 

\item What is more, in the previous case,  $g$ will be complete when $|\nu_M(p)|$ is bounded by  $G(r(p))$ for  a continuous function $G>0$ and {\em  no other hypothesis}.  

To prove this, otherwise, there would exist an incomplete unit minimizing geodesic $\gamma: [0,b)\in \Sigma$, where $b<\infty$, $\gamma(s)=(x(s),t(s), \gamma(0)=o$ and, along $\gamma$, $|\nu_M|$ is bounded by the maximum of $G([0,b])$. Then, both  $x'(s)$ and $t'(s)$ are bounded for $s\in [0,b)$ which permits the extensibility of $x(s)$ (by the completeness of $g_M$) and, then, of $\gamma$, through $b$.

\end{enumerate}
\end{remark}

\subsection{Application to  (codimension 1) solitons}\label{s3_2b}
Completeness will be an essential property to ensure Omori-Yau principle and, then, inexistence results. 
The natural hypotheses on $G$ stated in (a), (b) of Theorem~\ref{t_completeness} and in (c) of 
Corollary~\ref{c_complgraph} are related to other hypotheses  commonly used to ensure the Omori-Yau principle, specifically, those in \cite[formula (2.27)]{AMR} (which are the hypotheses (i)---(iii) below). Let us start clarifying that, even though the latter hypotheses seem a priori more restrictive than ours, they turn out equivalent.
\begin{proposition}\label{p_iii}
    Let $G_0:[0,\infty)\rightarrow \R $ be a continuous function satisfying: (a)~$G_0>0$, (b) $\int_0^\infty ds/G_0(s)=\infty$, (c) $G_0$ is nondecreasing. Then, there exists a $C^1$ function 
    $G:[0,\infty)\rightarrow \infty$ satisfying $G\geq G_0$ and:
    \begin{equation}
        \label{e_iii}
       (i) \; G(0)>0, \quad (ii) \; G'(s)\geq 0, \forall s\in [0,b), \quad (iii) \; \frac{1}{G}\not\in L^1([0,\infty)].
    \end{equation}
\end{proposition}

Note: clearly \eqref{e_iii} implies (a), (b), (c), being (iii) just a rewriting of (b). 
    
\begin{proof} This is a consequence of the $C^0$ density of $C^1$ functions into the $C^0$ ones, anyway, a simple construction goes as follows.
    Consider the polygonal function  $P_0$ with segments connecting the points $(k,G_0(x_{k+1}))$, for integers $ k\geq 0$. Notice that  $P_0\geq G_0$ and it satisfies (ii) (including its lateral derivatives). It also satisfies  (iii), as $\int_k^{k+1}dt/P_0(t)\geq \int_{k+2}^{k+3}dt/G_0(t)$, thus, the required inequality holding by summing in $k$. 
    Easily, the polygonal can be  smoothed at each $k$ maintaining (ii). Choosing a smoothed function bounded by $2G$, (iii) (as well as the lower bound by $G_0$) also will hold.     
\end{proof}

\begin{remark}\label{r_trivial} The following trivial facts will be used when necessary.

    If some elements are bounded by some $G_0(r)$ as above then they are also bounded for $G(r)$; this can be used when ensuring completeness. 
    
    Moreover, next these elements will be contained in a complete manifold where the parameter $r$ measures distance to a point. So, to  bound them for $r$ greater than a constant $R_0>0$ will be equivalent too. That is, one can replace $0$ by $R_0$ in the hypotheses (a)--(c) as well as in (i)--(iii). Finally, if $G$ satisfies any of these properties, then so does $A \, G$ with $A>1$. 
\end{remark}

To obtain the announced application for solitons, notice first:

\begin{lemma} \label{l_H}
    Let $F: \Sigma^n \rightarrow \mt^{n+1}_{-1}$ be a 
    $c$-translating soliton in $(\mt, \gt)$. Then, 
    $$
    \frac{H^2}{c^2} \leq |\nu_E|^2 < 2\frac{H^2}{c^2} .$$
\end{lemma}

\begin{proof} Putting 
$\nu=-\gt(\nu,\partial_t)\partial_t +\nu_M $ 
one has 
$g_M(\nu_M,\nu_M)=\gt(\nu,\partial_t)^2-1$ and
$$
|\nu_E|^2=\gE(\nu,\nu)=\gt(\nu,\partial_t)^2+g_M(\nu_M,\nu_M)= 2 \gt(\nu,\partial_t)^2-1.
$$
Now, apply that, for a soliton,  $\frac{H^2}{c^2}=\gt(\nu,\partial_t)^2$; for the lower bound, use $\gt(\nu,\partial_t)\geq 1$, which follows from the Lorentzian reverse triangle inequality for timelike vectors.   
\end{proof}
Summing up, finally:
\begin{corollary} \label{co:completeness}
    Let $\left(M, g_M\right)$ be a complete Riemannian manifold, consider the Lorentzian product $\mt=M \times \mathbb{R}$, $\gt=g_M-d t^2$, 
and let $\Sigma$ be a properly immersed translating soliton with mean curvature $H$. 

(1) Let $r_E$ and $r$ be the distances associated to $g_E$ and $g$  from a point $o\in \Sigma$. $(\Sigma,g)$ is complete if any of the following conditions hold
\begin{itemize}
    \item[(1a)] For a function $G$ satisfying (a) and (b) in Prop. \ref{p_iii}
$$
H \leq G\left(r_E(x)\right). \qquad \forall x \in \Sigma ,
$$
\item[(1b)] For any positive function $G$
$$
H \leq G\left(r(x)\right). \qquad \forall x \in \Sigma .
$$
\end{itemize}

(2) Let $r_M$ be the $g_M$-distance  from a point $o_M\in M$.

\begin{itemize}
    \item[(2a)] When $\Sigma$ is a graph, it is  complete  if   
\begin{equation*}\label{e_2a}
    H \leq G\left(r_M(x)\right) \qquad \forall x \in M
\end{equation*}
for  $G$ satisfying (a), (b) and (c) (or  (i)--(iii)) in Prop. \ref{p_iii}. 

\item[(2b)]
When $\Sigma$ is not a graph, it is complete if the graph of the soliton associated to its universal covering $\tilde M$ (according to \eqref{e_universal}) satisfies (2a).
\end{itemize}

\end{corollary}

\begin{proof} Using Lemma \ref{l_H}, the bounds for $H$ provide the bounds for $|\nu_E|$ which permit to apply Theorem \ref{t_completeness} (for the part (1a)), Remark \ref{r_caca} (for the part (1b)) and Corollary \ref{c_complgraph}  (for the part (2a), while (2b) is straightforward using Proposition~\ref{p_cover}(b)). 
\end{proof}

\subsection{A relation between $G$-bounds in terms of  $r_M$ and  $r$}\label{s_extra}
 When the hypersurface $\Sigma$ is a graph, the  facts that $r \leq r_M$ and $G$ is nondecreasing yield 
\begin{equation}\label{e_extra0}
    H \leq (G\circ r) \Longrightarrow H \leq (G\circ r_M). 
    \end{equation}
Thus, the second bound, which may be easier to estimate from a practical point of view, is  insufficient when the first one is sharp. Next, given a function $G$ which may be suitable for a sharp estimate in terms of $r$, we will give a new function $G^M$ that provides the natural estimate in terms of $r_M$. This can be used to facilitate the applicability of the Omori-Yau principle to be proved in the next section. 

\begin{proposition}\label{p_extra}
            Let $\Sigma$ be the graph in $(M\times \R,\bar g=g_M-dt^2)$  of a smooth function $u$. If its  (future-directed) normal $\nu=(\nu_M,\nu_{\R})$ satisfies:
    \begin{equation}\label{e_extra1}
            \nu_{\R}(x)  \leq G(r(x)), \qquad \forall x\in M,
        \end{equation}
    for $G$ satisfying (a), (b), (c) in Prop. \ref{p_iii} then, with our natural identifications,
    \begin{equation}\label{e_extra1bis}
            g \geq g_G, \qquad \hbox{where} \quad g_G:= \frac{g_M}{G^2\circ r_M}.
        \end{equation}
    
    Moreover, if $g_M$ is complete 
    then $g_G$ is complete too.
    In this case, the radial distance $d_G$ from $o_M\in M$ associated to $g_G$ satisfies
\begin{equation} \label{e_extra2}         r_G(x)=\int_0^{r_M(x)}\frac{dr}{G(r)}, \end{equation} and $r_G\leq r\leq r_M$.
    
\end{proposition}

\begin{proof}
    As $\nu=(\nabla^M u,1)/\sqrt{ 1-|\nabla^M u|^2(x)})$,  the bound \eqref{e_extra1} 
    on $\nu_\R$ yields
    \begin{equation}
        \label{extra3}
    1-|\nabla^M u|^2(x)\geq \frac{1}{G^2(r(x))}.
        \end{equation} 
    Thus, writing 
    any non-zero tangent vector at $(x,u(x))$ as $v=(v_M,v_{\R})$: 
$$\begin{array}{rl}
g(v,v)=     &  g_M(v_M,v_M)-du(v_M)^2=g_M(v_M,v_M)\left(1-g(\nabla^M u, \frac{v_M}{|v_M|})^2\right) \\
     \geq  &  g_M(v_M,v_M)\left(1-
     |\nabla^M u|^2(x)\right)\geq \frac{1}{G(r(x))^2} g_M(v_M,v_M)\\
     \geq  &  \frac{1}{G(r_M(x))^2} g_M(v_M,v_M)=g_G(v_M,v_M) \equiv g_G(v,v),
\end{array}
$$
the last two inequalities using, respectively, \eqref{extra3} and \eqref{e_extra0} (for the latter, put $H=G\circ r$). This proves \eqref{e_extra1bis} and the completeness of $g_G$ follows from \cite[Prop. 3.1]{EFSZ}.  

Notice that $g_G$ and $g_M$ are conformal with a radial factor and, then, any unit $g_M$ minimizing geodesic $\gamma$ from $o_M$ to $x$ will be a $g_G$ minimizing pregeodesic (i.e., a geodesic up to a reparametrization), with length given precisely by \eqref{e_extra2}. 
The last assertion follows because $g_G\leq g\leq g_M$.\footnote{\label{f_G} Even though the first inequality has been proven in \eqref{e_extra1bis}, it is worth pointing out that the inequality \eqref{e_extra1} assumes implicitly $G\geq 1$, as $-1=g(\nu,\nu)=g_M(\nu_M,\nu_M)-\nu_{\R}^2$. In the case of solitons,  $H=c \nu_\R $ and, then,  the bound of $H$ by $G$ is implicitly larger than $c$.}
\end{proof}

Now, a useful criterion can be stated.

\begin{corollary}
     For $G$,  $r$ and $r_M$ as in Prop. \ref{p_extra} with complete $g_M$, let
\begin{equation}\label{d_GM}
    G^M(s):= G\left(\int_0^s \frac{dr}{G(r)}\right), \qquad \forall s\geq 0.
\end{equation}
Then $G^M\leq G$, $G^M$ satisfies the conditions (a), (b), (c) in Prop. \ref{p_iii} and, moreover,
\begin{equation}\label{e_GM}
 G^M(r_M(x))  \leq G(r(x)) \qquad \qquad \forall x\in M.
\end{equation}
\end{corollary}
\begin{proof} As $G\geq 1$ (see footnote \ref{f_G}) and non-decreasing, $G^M\leq G$, and clearly (a)--(c) hold. Then, using \eqref{e_extra2},  $r_G\leq r$ and, again,   $G$ is non-decreasing: $G^M(r_M(x))=G(r_G(x)\leq G(r(x))$.
\end{proof}

\begin{remark}[Criterion for $r_M$ bounds]\label{r_GM} We will use the previous corollary as a criterion to bound  any  function $H\geq 0$ on $M$, namely:
$$
H(x) \leq G^M(r_M(x)) \quad \Longrightarrow \quad H(x) \leq G(r(x)). 
$$
In particular, choosing $G$ affine, say,  $G(r)=AB r + B$ with $A,B>0$,  a logarithm growth is achieved $G^M(s)=\ln (Ar+1) +B$, and others can be obtained by using  functions as in Remark \ref{r_BE}.
    \end{remark}

\section{{Omori-Yau maximum principle for $\Delta_{ch}$ with unbounded $H$} 
}\label{s4}

    Given a Riemannian manifold {$(\Sigma,g)$}  
and a smooth function $$ h: M \rightarrow \R,$$  the {\em drift Laplacian associated to $h$} is the operator 
$\Delta_h: C^{2}(M) \rightarrow C^0(M),$ given as
\begin{equation} \label{def:drift}
\Delta_hu:= \Delta u - g(\nabla h,\nabla u).
\end{equation}
 Analogously, if $\Ric \equiv \Ric_\Sigma$ is the Ricci tensor of $\Sigma$, the {\em Bakry-Émery} Ricci tensor is \begin{equation}
    \label{def:rich} \Ric_{h} := \Ric_\Sigma +  \nabla^2 h.
\end{equation}
\begin{definition} $(M,g)$ satisfies the {\em Omori-Yau maximum principle for $\Delta_h$} if for any $u \in C^2 (M)$ such that $u^*=\sup_M u < \infty$, then there exists  an {\em Omori-Yau sequence} $\{ x_k \}_{k \in \n}$, defined by the properties:
\begin{eqnarray}
    u(x_k) & \to & u^* , \nonumber \\
    |\nabla u(x_k)| & \to & 0, \label{eq:omori-yau} \\
    \liminf \Delta_h(u)(x_k)& \leq  &0 \nonumber 
\end{eqnarray}
\end{definition}
Next, our aim it to prove the following theorem, that  implies Theorem II { (taking into account the criterion for $r_M$ bounds  and other details  in Remarks \ref{r_GM}, }\ref{r_trivial})  as well as  other assertions in the introduction.
\begin{theorem}  \label{th:A} Let $(M,g_M)$ be a complete Riemannian manifold with sectional curvature bounded from below,   Sec$_M\geq -\kappa,$ with $\kappa\geq 0$. Consider the product manifold $\mt=M \times \R$ equipped with the metric $\gt= g_M-dt^2$. 

Let $F: \Sigma \to \mt$ be a  spacelike $c$-translating soliton of the mean curvature flow (not necessarily graphical). 
The Omori-Yau principle holds for the drift Laplacian $\Delta_{c \, h}$ on $(\Sigma,g)$ (where $g=F^\ast (\gt)$ and $h$ is the height function in Definition \ref{def:h}) if    the mean curvature satisfies
\begin{equation}  \label{eq:desH} H(x) \leq G(r(x)) \; , \forall x \in \Sigma \; , \end{equation}
where $G \in C^1([0,+\infty[)$ such that 

(i) $G(0)>0$, 
(ii) $G'(t) \geq 0$,
(iii) $\frac{1}{G(t)} \notin L^1(]0,+\infty[)$.\newline
Moreover, when $\kappa=0$, the principle holds if \eqref{eq:desH} is satisfied for a continuous function $G$, that is, whenever   $H$ is bounded in terms of the $d$-distance. 
\end{theorem}
For the proof, first the following lemmas yield lower bounds for the { Bakry-Émery} Ricci and drift Laplacian associated to $c\,h$ in terms of $G$.

\begin{lemma}\label{lema_driftRicci} Under the assumptions of Theorem \ref{th:A},
\begin{equation} \label{eq:vip-2}
    \Ric_{c \, h} (X,X) \geq - \, \frac{n\kappa}{c^2}  G^2 |X|^2.
\end{equation}    
\end{lemma}

\begin{proof}  
    Let us consider $X, Y, Z, W \in \mathfrak{X}(\Sigma)$. The Gauss equation for the immersion $F: \Sigma \to \mt$ and shape operator $A$ is given by
\begin{equation}
 R(X,Y,Z,W)= \rt(X,Y,Z,W) {-} \gt(A X,Z) \gt(A Y,W) {+}\gt(A X,W) \gt(A Y,Z) \;,
\end{equation}
using our previous conventions (see Remark \ref{r_Codazzi}). 

Let $\{E_1, \ldots,E_n\}$ be an orthonormal frame of $\mathfrak{X}(M)$ and $\nu$ an unit normal vector field in the same cone as $\partial_t$. Then, taking the trace in the Gauss equation and \eqref{eq:traceA}:
\begin{multline}
    \Ric_\Sigma(X,X) =\ricct(X,X) + \rt(X,\nu,X, \nu) 
- \operatorname{trace}(A) \, g(AX,X) + |AX|^2  \\ =
\ricct(X,X) + \rt(X,\nu,X, \nu)
+ H \, g(AX,X) {+} |AX|^2 \label{eq:534}
\end{multline}
(recall \eqref{eq:traceA}). Let us write  \begin{eqnarray}
    X & = & X_M - \gt(X, \partial_t)\partial_t, \label{eq:XX} \\
    \nu & = & \nu_M - \gt(\nu,\partial_t) \partial_t = \nu_M + (\varepsilon H/c) \partial_t, \label{eq:NN}\end{eqnarray}
the latter from the equation of the translator \eqref{soliton-scalar}.
Using
this in $\rt(X,\nu,X,\nu)$ one obtains sixteen terms, all vanishing but  the one not applied to  $\partial_t$ 
(due to the product structure, 
see for instance \cite[Ch. 7]{Oneill}), 
thus
\begin{multline} \label{eq:536}
    \rt(X, \nu,X,\nu)= \rt \left(X_M,\nu_M,X_M,\nu_M\right)= \\
    \Sec_M\left(X_M,\nu_M\right) \left[ g_M(X_M,X_M) \, g_M(\nu_M,\nu_M)- g_M(X_M,\nu_M)^2\right].
\end{multline}

Next,   apply the bound   $\Sec_M \geq - \kappa$, $(\kappa \ge 0) $  in \eqref{eq:536} plug it in \eqref{eq:534} and use also 
$$\Ric_M \geq - (n-1) \, \kappa \, g_M.$$
    to deduce 
\begin{multline}
\label{eq:537}
\Ric_\Sigma(X,X) \geq - (n-1) \, \kappa \, |X_M|^2  +H \, g(A X,X)  +  |A X|^2 \\
- \kappa \left[ |X_M|^2|\nu_M|^2- g_M (X_M,\nu_M)^2
\right]
\end{multline}
Using here the definition of $\Ric_{c \, h}$ (formula \eqref{def:rich}) and \eqref{lem:h}, one has:

\begin{multline} \label{eq:538}
\Ric_{c \, h} (X,X) \geq - (n-1) \, \kappa \, |X_M|^2+|A X|^2 \\
- \kappa \left[ |X_M|^2|\nu_M|^2- g_M (X_M,\nu_M)^2
\right].
\end{multline}
Taking into account \eqref{eq:XX} and \eqref{eq:NN}, and using \eqref{soliton-scalar} and \eqref{eq:funcion-h}, one has:
\begin{equation} \label{eq:NN-1}
    |X_M|^2|\nu_M|^2= \left(\frac{H^2}{c^2}-1 \right) \left(|X|^2+\frac 1{c^2} g(X, \nabla(c h)) \right), \quad 
\end{equation}
and 
\begin{multline} \label{eq:XX-1}
    \gt(X_M,\nu_M) = \gt (X+\gt(X,\partial_t) \partial_t, \nu+\gt(\nu, \partial_t) \partial_t) =   \\
    =\frac{\varepsilon H}{c} \gt(X,\partial t) =\frac{H}{c^2} g(X, \nabla (c h)). 
\end{multline}
Using \eqref{eq:NN-1} and \eqref{eq:XX-1} in \eqref{eq:538}, one has:
\begin{multline} \label{eq:538-b}
\Ric_{c \, h} (X,X) \geq - (n-1) \, \kappa \, |X_M|^2+|A X|^2 \\
- \kappa \left[ \left( \frac{H}{c}\right)^2|X|^2-|X|^2-\frac{H}{c^2}   g (X,c\nabla h)^2
\right].
\end{multline}
Now, let us bound $|X_M|$ using \eqref{eq:funcion-h}:
$$ |X_M|^2-|X|^2=g (X, \nabla h)^2 \leq |X|^2 \left(\frac{H^2}{c^2}-1 \right),$$
in other words, we have 
 \begin{equation}\label{eq:X_M}
     |X_M|^2 \leq |X|^2 \frac{H^2}{c^2}.
 \end{equation}
Using  
\eqref{eq:X_M} in \eqref{eq:538-b}, we get 
\begin{multline} \label{eq:vip}
    \Ric_{c \, h} (X,X) \geq - (n-1) \, \kappa \, \left(\frac{H}{c} \right)^2|X|^2+|A X|^2\\
- \kappa \left( \frac{H}{c}\right)^2|X|^2+ \kappa |X|^2+ \kappa \frac{H}{c^2}   g (X,c\nabla h)^2 \geq\\
    -n \, \kappa \left( \frac Hc\right)^2 |X|^2.
\end{multline}
So, \eqref{eq:vip-2} follows using the $G$-bound \eqref{eq:desH} in the hypotheses of our theorem.
\end{proof}
\begin{remark}[The case $\kappa=0$ and completeness] \label{re:k=0}
     If $\kappa=0$, from \eqref{eq:538-b} we immediately deduce that
    $\Ric_{c \, h}(X,X) \geq |AX|^2 \geq 0.$ Then, whenever $\Sigma$ is complete, $\Delta_{c h}$ satisfies the Omori-Yau maximum principle over $\Sigma$; indeed, general results by Chen and Qiu (Theorem~1 and Remark 1 in \cite{CQ}) are applicable. 
    
    Corollary \ref{co:completeness}  ensures the completeness of $\Sigma$ in Theorem \ref{th:A}. In particular, when $\kappa=0$, the item 
    (1b) of this corollary (coming from Remark \ref{r_caca}(3)) implies completeness whenever the bound \eqref{eq:desH} holds for a (necessarily non-negative) continuous function $G$, even dropping the hypotheses (ii) and (iii) therein. 
\end{remark}
Taking the previous remark into account, we will assume from now on that $$\kappa>0.$$
 The next steps are inspired in   \cite[Theorem 8.1]{AMR}.
\begin{lemma}  \label{lema_driftLaplacian}  Under  Theorem \ref{th:A}, there exists a constant $C>0$ such that
$$
\Delta_{c h} r(x) \leq  C \, G(r(x)).
$$
weakly on all $\Sigma$ minus the ball $\bar{B}(o,2)$,  and smoothly outside the cut locus $\operatorname{cut}(o)$.
\end{lemma}

\begin{proof}
From Lemma \ref{lema_driftRicci}, for any  $x \in \Sigma$ and  $v \in T_x\Sigma, $ one has:
\begin{equation} \label{eq:vip-3} 
    \Ric_\Sigma  (v,v) \geq -\frac{n\kappa}{c^2} \,  G^2(r(x)) |v|^2 -\nabla^2(c \, h)(v,v), 
     \end{equation}

\smallskip
     
\noindent {\em Step 1. Ricatti inequality for $\Delta \varphi$ on radial geodesics}. Applying Bochner formula to the distance function $r$ to the chosen point $o$ in $\Sigma\setminus (\{o\} \cup \cut(o))$,                                           
\begin{equation} \label{eq:bochner}
    0= | \nabla^2r|^2+\Ric_\Sigma(\nabla r,\nabla r) +g(\nabla \Delta r, \nabla r).
\end{equation}
Fix $x \in \Sigma\setminus (\{o\} \cup \cut(o))$ and consider a minimizing geodesic $$\gamma:[0,l] \rightarrow \Sigma, \quad l={\rm length}(\gamma),$$
such that $\gamma(0)=o$ and $\gamma(l)=x.$ Let us define $$\varphi(t):=(\Delta r \circ \gamma)(t), \quad t\in [0,l].$$
Then, using \eqref{eq:bochner}, we deduce 
\begin{equation} \label{eq:bochner-2}
    | \nabla^2r(\gamma(t))|^2+\Ric_\Sigma(\gamma'(t),\gamma'(t)) +\varphi'(t)=0.
\end{equation}
Taking into account that  $(\Delta r)^2\leq n |\nabla^2 r|^2,$ (indeed, the constant $n-1$ 
as in \eqref{e_radial_laplacian} also works,  but this is not  relevant here), we use \eqref{eq:bochner-2} and \eqref{eq:vip-3} to get the following Riccati inequality:
\begin{equation}
    \label{eq:bochner-3}
    \varphi'(t)+\frac 1n \varphi^2(t) \leq -\Ric_\Sigma(\gamma'(t),\gamma'(t)) \leq \frac{n\kappa}{c^2} \, G^2(t)+\nabla^2(c \, h) (\gamma'(t),\gamma'(t)),
\end{equation}
where $G(t)=G(r(\gamma(t))).$
\vskip 5mm
     
\noindent {\em Step 2. Inequality in terms of the ODE  solution $w$}.
Put $ \bar G^2(t)  := \frac{\kappa}{c^2} \, G^2(t)$ and consider the  Cauchy problem:
\begin{equation} \label{eq:Cauchy}
\left\{\begin{array}{l}
w^{\prime \prime}-  \bar G^2(s) w=0, s \in \mathbb{R}_0^{+} \\
w(0)=0, w^{\prime}(0)=1
\end{array}\right.
\end{equation}
From the boundary conditions,  there  exists $\epsilon>0$ such that $w, w'>0$ on $(0,\epsilon)$, thus, $w''> 0$ and  
$w, w', w''>0$ on $\mathbb{R}^{+}$. 
Notice that the   expansion of $\Delta r$ at $o$  (see for instance \cite[formula (1.226)]{AMR} gives: 
\begin{equation}\label{e_radial_laplacian}
\Delta r=\frac{n-1}{r}+O(r).    
\end{equation}
    However, around $o$, one has $w(r)= r+O(r)$ and, then, \begin{equation} \label{eq:laplacianr}
        \lim_{x \to o} w^2(r) \Delta r=0.
    \end{equation}

Using \eqref{eq:bochner-3},  
\begin{equation} \label{eq:bb}
    \begin{aligned} 
\left(w^2 \varphi\right)^{\prime} & =2 w w^{\prime} \varphi+w^2 \varphi^{\prime}  \\ 
& \leq 2 w w^{\prime} \varphi-\frac{w^2}{n} \varphi^2 + n \,w^2  
{\bar G}^2(t) +  w^2\nabla^2(c h)(\dot{\gamma}, \dot{\gamma})
\end{aligned}\end{equation}
By definition, $\nabla^2(c h)(\dot{\gamma}, \dot{\gamma})=g\left(\nabla_{\dot{\gamma}} \nabla c h, \dot{\gamma}\right)=\dot{\gamma} (g(\dot{\gamma}, \nabla c h))=(g(\dot{\gamma}, \nabla c h))'$, thus,

\begin{equation} \label{eq:bb-2}
\begin{aligned}
\left(w^2 \varphi\right)^{\prime} &  \leq 2 w w^{\prime} \varphi-\frac{w^2}{n} \varphi^2  + n \, w^2 \bar G^2(t) + w^2 (g(\dot{\gamma}, \nabla c h))^{\prime}  \\
& =-\left(\frac{w \varphi}{\sqrt{n}}-\sqrt{n} w^{\prime}\right)^2+ 
n \left(w^{\prime}\right)^2  + n \, w^2 
\, { \bar G}^2(t)
+w^2 (g(\dot{\gamma}, \nabla c h))^{\prime}.
\end{aligned}
\end{equation}
Let  $\displaystyle \varphi_{\bar G}(t):=n \frac{w'(t)}{w(t)}$ on $(0,l]$. Using the equation satisfied by $w$, we have
$$ (w^2 \varphi_{\bar G})'= (n \; w w')'= n(w')^2+n \; w w''= n \left( w' \right)^2+ n \, w^2 \bar G^2(t). $$
Hence, inequality \eqref{eq:bb-2} implies
$$ (w^2 \varphi)'  \leq (w^2 \varphi_{\bar G})' + (g(\dot{\gamma}, \nabla c h))' \, w^2.$$
Taking into account \eqref{eq:laplacianr} and integrating on $[0, r]$, one gets
\begin{equation} \label{eq:5321}
    w^2(r) \varphi(r) \leq w^2(r) \varphi_{\bar G}(r)+\int_0^r(g(\dot{\gamma}, \nabla c h))^{\prime} w^2 d t
\end{equation}

\smallskip
     
\noindent{{\em Step 3. Inequality for the drift Laplacian}. } Define $\varphi_{c h}:=\left(\Delta_{c h} r\right) \circ \gamma=\varphi-$ $g(\nabla c h, \dot{\gamma}) \circ \gamma$, using \eqref{eq:5321} and integrating by parts, we get
\begin{equation} \label{eq:5321b}
\begin{aligned}
w^2(r) \varphi_{c h}(r) \leq & w^2(r) \varphi_{\bar G}(r)- w^2(r) \, g(\nabla c h, \dot{\gamma}) \circ \gamma(r)+\int_0^r(g(\dot{\gamma}, \nabla c h))^{\prime} w^2 d t \\
& =w^2(r) \varphi_{\bar G}(r)-\int_0^r\left(w^2\right)^{\prime} g(\nabla c h, \dot{\gamma}) d s
\end{aligned}
\end{equation}
From \eqref{eq:funcion-h}, $|\nabla c h|^2=H^2-c^2$,  thus $|\nabla c h|^2 \leq G^2 = \frac{c^2}{\kappa} \cdot \bar G^2$. Then, by Cauchy-Schwarz inequality, we have
$$
-g(\nabla c h, \dot{\gamma}) \circ \gamma \leq|\nabla c h| \circ \gamma \leq  \frac{c}{\sqrt{\kappa}} \,  \bar G \circ \gamma .
$$
Since $w, w'\geq 0$, we have that $\left(w^2\right)^{\prime}=2 w w^{\prime} \geq 0$. Thus, using this inequality, \eqref{eq:5321b} and integrating again by parts, we get the following one,
$$
\begin{aligned}
w^2(r) \varphi_{c h}(r) & \leq w^2(r) \varphi_{\bar G}(r)+ \frac{c}{\sqrt{\kappa}} w^2(r) \bar G(r)-\frac{c}{\sqrt{\kappa}}\int_0^r w^2 \bar G^{\prime} d s \\
& \leq w^2(r) \varphi_{\bar G}(r)+\frac{c}{\sqrt{\kappa}} w^2(r) \bar G(r) ;
\end{aligned}
$$
on $(0, l]$. This is equivalent to
$$
\varphi_{c h}(r) \leq \varphi_{\bar G}(r)+\frac{c}{\sqrt{\kappa}} \bar G(r), \forall r \in(0, l] .
$$

As this is valid for any minimizing geodesic joining $o$ with any point $x \in \Sigma \backslash(\{o\} \cup \operatorname{cut}(o))$ we have
\begin{equation} \label{eq:laplacian-2}
\Delta_{c h} r(x) \leq n \, \frac{w^{\prime}(r(x))}{w(r(x))}+ \frac{c}{\sqrt{\kappa}} \bar G(r(x)) \quad \mbox{on $\Sigma \backslash(\{o\} \cup \operatorname{cut}(o))$.}
\end{equation}

\smallskip
     
\noindent {{\em Step 4. Inequality in terms of $\bar G$ outside a ball.}} Let
$$
f(t)=\frac{1}{\bar G(0)}\left(\exp \left\{\int_0^t \bar G(s) d s\right\}-1\right) ;
$$
for $t \geq 0$,  
which satisfies $f(0)=0$,  $f^{\prime}(0)=1$,  and  

$$
\begin{aligned}
f^{\prime \prime}(t)-{\bar G}^2(t) f(t) 
=\frac{{\bar G}^{\prime}(t)}{{\bar G}(0)} \exp \left\{\int_0^t {\bar G}(s) d s\right\}+\frac{{\bar G}^2(t)}{{\bar G}(0)} \geq 0 .
\end{aligned}
$$
Taking into account that $w$ is a solution of the equality (see \eqref{eq:Cauchy}) and that $w>0, f>0$ on $(0,+\infty)$, 
we have that 
$w (f''-{\bar G} \, f) \geq 0$, 
$-f (w''-{\bar G} \, w)=0$ and,  therefore, $f^{\prime \prime} w-w^{\prime \prime} f \geq 0$. This is the derivative of the function $f^{\prime} w-w^{\prime} f$ and given that $f^{\prime}(0) w(0)-$ $w^{\prime}(0) f(0)=0$, we conclude that $f^{\prime} w-w^{\prime} f \geq 0$ on $[0,+\infty)$. In particular, for $t>0$, we have that
\begin{equation} \label{eq:caos}
\frac{w^{\prime}(t)}{w(t)} \leq \frac{f^{\prime}(t)}{f(t)}=\frac{\frac{{\bar G}(t)}{{\bar G}(0)} \exp \left\{\int_0^t {\bar G}(s) d s\right\}}{\frac{1}{{\bar G}(0)}\left(\exp \left\{\int_0^t {\bar G}(s) d s\right\}-1\right)} .
\end{equation}
{The required bound of the left-hand side in terms of $\bar G(t)$ (remedying also   that both sides diverge as $t \rightarrow 0$) is achieved noticing that $\exp \left\{\int_0^t {\bar G}(s) d s\right\} (\geq 1)$ is increasing, so that beyond $t>1$, \eqref{eq:caos} yields}
\begin{equation}\label{eq:lambda}
  \frac{w^{\prime}(t)}{w(t)} \leq \frac{f^{\prime}(t)}{f(t)} \leq {\bar G}(t) \frac{\exp \left\{\int_0^t {\bar G}(s) d s\right\}}{\left(\exp \left\{\int_1^t {\bar G}(s) d s\right\}-1\right)}, \quad \forall t>1.  
\end{equation}

As $\lambda(t):=\frac{\exp \left\{\int_0^t {\bar G}(s) d s\right\}}{\left(\exp \left\{\int_1^t {\bar G}(s) d s\right\}-1\right)}$
is decreasing there,  we get:
\begin{equation} \label{eq:55-55}
\frac{w^{\prime}(t)}{w(t)}  \leq \lambda(2) \, \cdot \, {\bar G}(t), \qquad  \, \forall t \geq 2.
\end{equation}
 Finally, using \eqref{eq:55-55} in \eqref{eq:laplacian-2}, we obtain 
$$
\Delta_{c h} r(x) \leq \left(\frac{c}{\sqrt{\kappa}}+ n \, \lambda(2)\right) {\bar G}(r(x)) ;
$$
on $\Sigma \backslash\left(\operatorname{cut}(o) \cup \bar{B}(o,2)\right)$. As the cut locus has measure zero in $\Sigma$, then the previous inequality also holds  weakly on $\Sigma \backslash \bar{B}(o,2)$. Thus, taking $C=1+n \lambda(2) \frac{\sqrt{\kappa}}{c},$
the required inequality holds.
\end{proof}


\begin{proof}[Proof of Theorem \ref{th:A}] {From Remark \ref{re:k=0} we can focus on the case $\kappa>0$ and assume that $\Sigma$ is complete (and non compact) and, applying Lemma \ref{lema_driftLaplacian} with $\bar G:= C \, G$,  
\begin{equation}\label{barG}
    \Delta_{c h} r(x)\leq {\bar G}(r(x)),    \qquad \forall x\in\Sigma \backslash\left(\{o\} \cup \operatorname{cut}(o) \cup \bar{B}(o,2)\right).
\end{equation}
}

     
\noindent {{\em Step 1. Bounds in terms of $\vartheta(     x)=\int_0^{r(x)} 1/G$}}.
Let start by the following definitions
$$
\phi(t):=\int_0^t \frac{d s}{{\bar G}(s)}, \qquad \forall t\geq 0, \qquad  \vartheta(x):=\phi(r(x)), \; \forall x\in \Sigma.
$$
We are interested in properties of $\theta$. By those of $\bar G$ inherited from $G$:

\begin{equation} \label{eq:Gons}
  \phi^{\prime}(t)=\frac{1}{{\bar G}(t)} \text { and } \phi^{\prime \prime}(t) \leq 0;  \qquad   \vartheta(x) \rightarrow+\infty \text { as } r(x) \rightarrow+\infty,
\end{equation}
the latter as $\frac{1}{{\bar G}(t)} \notin L^1(0,+\infty)$. 
By the chain rule, we get $\nabla \vartheta=\phi^{\prime}(r) \nabla r$ and
\begin{equation} \label{eq:Panisardi}
|\nabla \vartheta|=\frac{1}{{\bar G}(r)} \leq \frac{1}{{\bar G}(0)} \leq \Lambda:= \max \{1,1 / {\bar G}(0)\}.
\end{equation}
Finally, the following consequence of  \eqref{barG} holds by using \eqref{eq:Gons}:
\begin{equation} \label{eq:Panisardi-2}
\begin{aligned}
\Delta_{c h} \vartheta(x) 
& =\phi^{\prime}(r(x)) \Delta_{c h} r(x)+\phi^{\prime \prime}(r(x)) \\
& \leq \phi^{\prime}(r(x)) \Delta_{c h} r(x)=\frac{1}{{\bar G}(r(x))} \Delta_{c h} r(x) 
\leq 1 \leq \Lambda ,
\end{aligned}
\end{equation}

     
\noindent {{\em Step 2. Maximum for an adapted difference $u-\vartheta_{\varsigma}$.}}
Now, consider a function $u \in \mathrm{C}^2(\Sigma)$ such that $u^*:=\sup _M u<+\infty$. Fix $\eta>0$ and consider
\begin{equation}\label{e_eta}
\begin{gathered}
A_\eta:=\left\{x \in M: u(x)>u^*-\eta\right\}, \\
B_\eta:=\left\{x \in A_\eta:|\nabla u(x)|<\eta\right\} .
\end{gathered}
\end{equation}
Since $\Sigma$ is complete, from \cite[Prop. 2.2]{AMR}, $B_\eta \neq \varnothing$. Therefore, all we need to prove is
$$
\inf _{B_\eta} \Delta_{c h} u \leq 0 .
$$
Reasoning by contradiction, assume that there exists $\varsigma_0>0$ such that
\begin{equation} \label{eq;morillas}
    \inf _{B_\eta} \Delta_{c h} u \geq \varsigma_0>0 .
\end{equation}

First of all, this condition implies that $u^*$ cannot be attained at $M$. Otherwise, there exists a point $x_0 \in M$ such that $u^*=u\left(x_0\right), \nabla u\left(x_0\right)=0$ and $\Delta u\left(x_0\right) \leq 0$; as $\Delta_{c h} u\left(x_0\right)=\Delta u\left(x_0\right)$, a contradiction with \eqref{eq;morillas} appears.
Define
$$
\Omega_T=\{x \in M: \vartheta(x)>T\} .
$$
Its complement $\Omega^c_T:=\Sigma \backslash \Omega_T=\{x \in \Sigma: \vartheta(x) \leq T\}$ is closed, and it is also compact because it is bounded, as 
$\vartheta$ is a function that diverges radially. 
As a consequence, for all $T>0, u$ attains its maximum in $\Omega_T^c$, namely:
$$u_T^*:= \max_{\Omega_T^c} (u).$$
As $\left\{\Omega^c_T\right\}$ is an exhaustion, that is, $\Omega_{T_1}^c \subset \Omega_{T_2}^c$ for $T_1<T_2$ and $\Sigma \equiv \bigcup_{T>0} \Omega_T^c$, there exists a divergent sequence $\left\{T_j\right\}_{j \in \mathbb{N}}$ such that
\begin{equation} \label{eq:araña}
    u_{T_j}^* \rightarrow u^* \text { as } j \rightarrow+\infty \text {. }
\end{equation}
Then, we can take $T_1$ such that $u_{T_1}^*>u^*-\eta / 2$ and $\Omega_{T_1} \subset \Sigma \setminus \bar B (o,2)$, and notice that $u^*_{T_1}<u^*_{T_2}$ for some $T_2>T_1$. Now choose $\alpha, \delta>0$ such that
\begin{equation}\label{e_delta}
u_{T_1}^*<\alpha<\alpha+\delta<u_{T_2}^*.
\end{equation}

Next, for $\varsigma>0$, we define
$$
\vartheta_{\varsigma}(x):=\alpha+\varsigma\left(\vartheta(x)-T_1\right).
$$
Our aim is to prove that  the function $u-\vartheta_{\varsigma}$ attains a local maximum in  $\bar{\Omega}_{T_1} \backslash \Omega_{T_3}$, for a suitable choice of $\varsigma>0$ {and $T_3>T_2$.}

Firstly, it is obvious that $\vartheta_{\varsigma}(x)=\alpha$ on $\partial \Omega_{T_1}$ and
\begin{equation} \label{eq:gin-y-1}
\alpha \leq \vartheta_{\varsigma}(x) \leq \alpha+\varsigma\left(T_2-T_1\right) \text { on } \bar\Omega_{T_1} \backslash \Omega_{T_2} .
\end{equation}
Now, we  choose a small enough $\varsigma$ such that
$\varsigma\left(T_2-T_1\right)<\delta$ and, thus
\begin{equation} \label{eq:gin-y-3}
\alpha \leq \vartheta_{\varsigma}(x)<\alpha+\delta \text { on } \bar\Omega_{T_1} \backslash \Omega_{T_2} .
\end{equation}
Choosing a smaller $\varsigma$ if necessary, we can also ensure, for $\eta$ in \eqref{e_eta} and $\varsigma_0$ in \eqref{eq;morillas},
\begin{equation} \label{eq:gin-y}
\left\{\begin{array}{l}
\Delta_{c h}\left(\vartheta_{\varsigma}\right)(x)=\varsigma \Delta_{c h}(\vartheta)(x) \leq \varsigma \Lambda<\varsigma_0, \\
\left|\nabla \vartheta_{\varsigma}\right|=\varsigma|\nabla \vartheta| \leq \varsigma \Lambda<\eta .
\end{array}\right.
\end{equation}


With these choices,  consider the function $u-\vartheta_{\varsigma}$. For any $x \in \partial \Omega_{T_1}$ we have
\begin{equation} \label{eq:gin-y-4}
\left(u-\vartheta_{\varsigma} \right)(x) \leq u^*_{T_1}-\alpha <0 .
\end{equation}
By the definitions of $\Omega_{T}$ and $u^*_T$, there exists $\bar{x}\in \Omega_{T_1} \backslash \Omega_{T_2}$ such that $u(\bar{x})=u_{T_2}^*$ and,  
\begin{equation} \label{eq:gin-y-44}
\left(u-\vartheta_{\varsigma}\right)(\bar{x}) \geq u_{T_2}^*-\alpha-\delta>0 ,
\end{equation}
the latter from \eqref{eq:gin-y-3}.
Moreover, as $\vartheta$ is divergent  and   $u^*<+\infty$, it is also possible to choose $T_3>T_2$  such that
\begin{equation} \label{eq:gin-y-55}
\left(u-\vartheta_{\varsigma}\right)(x)<0 \text { for } x \in \Omega_{T_3} .
\end{equation}
{(for example, $T_3\geq u^*$ suffices).}
Finally, we conclude from \eqref{eq:gin-y-4}, \eqref{eq:gin-y-44} and \eqref{eq:gin-y-55}  that $u-\vartheta_{\varsigma}$ attains a local maximum $x_0$ in 
the interior of the compact set $\bar \Omega_{T_1} \backslash \Omega_{T_3}$. 
Then, $\mu:=\left(u-\vartheta_{\varsigma}\right)\left(x_0\right)>0$
and
$$
u\left(x_0\right)=\vartheta_{\varsigma}\left(x_0\right)+\mu>\vartheta_{\varsigma}\left(x_0\right)>\alpha>u_{T_1}^*>u^*-\eta / 2 ,
$$
that is,  $x_0 \in A_\eta \cap \Omega_{T_1}$. 
\vskip 5mm

\noindent {{\em Step 3. Discussion of cases.}} At this point, we have to distinguish two cases.
 \vskip 3mm
 
 {\em CASE I:} $x_0 \notin \operatorname{cut}(o)$.
Since $x_0$ is a critical point, $\nabla\left(u-\vartheta_{\varsigma}\right)\left(x_0\right)=0$. Thus, by the linearity of the gradient and \eqref{eq:gin-y}, we have
$$
|\nabla u|\left(x_0\right)=\left|\nabla \vartheta_{\varsigma}\right|\left(x_0\right)<\varsigma \Lambda<\eta ,
$$
and $x_0 \in B_\eta$. Since $x_0$ is a maximum, $\Delta_{c h}(u-\vartheta_{\varsigma})\leq 0$, and using again \eqref{eq:gin-y}
$$
\Delta_{c h}(u)\left(x_0\right)\leq \Delta_h\left(\vartheta_{\varsigma}\right)\left(x_0\right)<\varsigma_0,
$$
in  contradiction with \eqref{eq;morillas}.
\vskip 3mm

{\em CASE II:} $x_0 \in \operatorname{cut}(o)$.
Next, the well-known {Calabi's trick is used} to bound $|\nabla u|$. Let us consider a minimizing geodesic $\sigma$ parametrized by the arc length joining $o$ with $x_0$. Let $o_{\varepsilon}:=\sigma(\varepsilon)$ for a small $\varepsilon>0$ such that   $r_{\varepsilon}(x)=\operatorname{dist}\left(x, o_{\varepsilon}\right)$ is smooth in a neighborhood of $x_0$ (see for example   \cite[Lemma 2.1]{AMR} or \cite[p. 284]{petersen}). Notice that 
$$
r(x)=\operatorname{dist}(o, x) \leq \left(o, o_{\varepsilon}\right) + \operatorname{dist}\left(o_{\varepsilon}, x \right)+\operatorname{dist}=r_{\varepsilon}(x)+\varepsilon,
$$
by the triangle inequality, with  equality on any point of $\sigma$, in particular, on $x_0$. Let
$$
\vartheta^\epsilon(x):=\phi\left(r_\epsilon(x)+\epsilon\right) .
$$
Since $\phi$ is increasing, we have
\begin{equation} \label{eq:rojo}
\vartheta(x)=\phi(r(x)) \leq \phi\left(r_\epsilon(x)+\epsilon\right)=\vartheta^\epsilon(x),
\end{equation}
in particular,  the equality holds for $x_0$. We define now
$$
\vartheta_{\varsigma}^\epsilon=\alpha+\varsigma\left(\vartheta^\epsilon-T_1\right)
$$
where $\alpha$ and $\varsigma$ are the same chosen for $\vartheta_\varsigma$ above. From \eqref{eq:rojo}, 
\begin{equation} \label{eq:rojo-1}
u(x)-\vartheta_{\varsigma}^\epsilon(x) \leq u(x)-\vartheta_{\varsigma}(x),
\end{equation}
on $\Sigma$. {Thus, in $\bar \Omega_{T_1} \backslash \Omega_{T_3}$:
\begin{equation} \label{eq:rojo-2}
u-\vartheta_{\varsigma}^\epsilon \leq \mu.
\end{equation}}
As  the equality in \eqref{eq:rojo} holds for $x_0$, the same happens {in \eqref{eq:rojo-2}, that is, 
$u-\vartheta_{\varsigma}^\epsilon$ attains  a local maximum at $x_0$}. Given that $x_0 \notin \operatorname{cut}\left(o_\epsilon\right)$, we can compute the gradient of $\vartheta_{\varsigma}^\epsilon$ in $x_0$. Since it is a local maximum,  $\nabla u\left(x_0\right)=\nabla \vartheta_{\varsigma}^\epsilon\left(x_0\right)$ and, then,

$$
\begin{aligned}
\left|\nabla u\left(x_0\right)\right| & =\left|\nabla \vartheta_{\varsigma}^\epsilon\left(x_0\right)\right|=\varsigma\left|\nabla \vartheta^\epsilon\left(x_0\right)\right|=\varsigma \phi^{\prime}\left(r_\epsilon\left(x_0\right)+\epsilon\right)\left|\nabla r_\epsilon\left(x_0\right)\right| \\
& \leq\frac{\varsigma}{\bar G\left(r_\epsilon\left(x_0\right)\right)} \leq \frac{\varsigma}{\bar G(0)}<\eta .
\end{aligned}
$$
Thus, $x_0 \in B_\eta$.

To obtain the required contradiction with the drift Laplacian, consider the set
$
K:=\left\{x \in \Omega_{T_1}:\left(u-\vartheta_{\varsigma}\right)(x)=\mu\right\},
$
which includes $x_0\in K$. For every $x \in K$, we have
$$
u(x)=\alpha+\varsigma\left(\vartheta(x)-T_1\right)+\mu>\alpha>u^*-\eta / 2,
$$
which implies that $K \subset A_\eta$. 
Choosing any regular value $\chi$ of $u-\vartheta_\varsigma$ satisfying  $0<\chi<\mu$ (thus $\left(u-\vartheta_{\varsigma}\right)(x_0)>\chi$) let  $\Lambda_{x_0}$ denote  the connected component of the set
$
\left\{x \in \Omega_{T_1}:\left(u-\vartheta_{\varsigma}\right)(x)>\chi\right\}
$
containing $x_0$. As $\Lambda_{x_0}$ is contained in $\bar{\Omega}_{T_1} \backslash \Omega_{T_3}$ (recall  \eqref{eq:gin-y-4} and \eqref{eq:gin-y-55}), it is  relatively compact. We know that { \eqref{eq:gin-y} holds weakly on $M \backslash \bar{B}(o,2)$} so, trivially, 
$$
\Delta_{c h}\left(\vartheta_{\varsigma}+\chi\right)<\varsigma_0 \leq \Delta_{c h} u \text { weakly on } \Lambda_{x_0} .
$$
Therefore, from the maximum principle, $u-\vartheta_{\varsigma}-\chi$ attains it maximum at the boundary. As $u-\vartheta_{\varsigma}-\chi=0$ on $\partial \Lambda_{x_0}$, so $u \leq \vartheta_{\varsigma}+\chi$ on $\bar{\Lambda}_{x_0}$. Nevertheless, since $x_0 \in \Lambda_{x_0}$, the previous inequality implies that $\mu \leq \chi$ which is the required contradiction.
\end{proof}

\section{Non-existence results} \label{sec:nonexistence} Next, we  prove  nonexistence results  for  complete spacelike translating solitons. We start  with a general non-existence result for manifolds with non-negative Ricci curvature. {It is contained in \cite[Theorem 4.2]{CQ} which develops ideas by 
Chen and Qiu in \cite[Theorem 3 and Remark 4]{CQ}, and the proof is included for the sake of completeness.}

\begin{proposition}[{\cite{BL,CQ}}]\label{th:nonexistence-1} Let $\left(M, g_M\right)$ be a complete Riemannian manifold with  non-negative Ricci curvature. Consider $\mt=M \times \mathbb{R}$ endowed with the metric $\gt=g_M-d \, t^2$. Then, there are no complete spacelike translating solitons satisfying the Omori-Yau principle for $\Delta_{c \, h}$. \end{proposition}

\begin{proof} We proceed by contradiction. Assume that there exists a translator $\Sigma \subset \mt$ with those conditions. Therefore, by hypothesis, the Omori-Yau maximum Principle for the drift Laplacian $\Delta_{c h}$ holds on $\left(\Sigma, g\right)$. Consider the function $$f(p)=-\frac{1}{\sqrt{1+H^2(p)/ c^2}}, \qquad \forall p\in \Sigma.$$ As $\sup f \leq 0$,  the Omori-Yau maximum principle is applicable and, easily, 
\begin{equation} \label{eq:william}
\Delta_{c h} f=\frac{\Delta_{c h} H^2(p)}{2 c^2\left(1+H^2 / c^2\right)^{3 / 2}}-\frac{3\left|\nabla H^2\right|^2}{4 c^4\left(1+H^2 / c^2\right)^{5 / 2}} .
\end{equation}

Using the hypothesis $\operatorname{Ric}_M \geq 0$ and the well-known inequality $H^2 \leq n|A|^2$ in Lemma \ref{lem:alias}, we get $\Delta_{c h} H \geq \frac{H^3}{n}$. Then,  a straightforward computation shows
$$
\Delta_{c h} H^2 \geq \frac{2}{n} H^4+2|\nabla H|^2 .
$$

Plugging this in \eqref{eq:william}, we have
$$
\begin{aligned}
\Delta_{c h} f & \geq \frac{\frac{2}{n} H^4+2|\nabla H|^2}{2 c^2\left(1+H^2 / c^2\right)^{3 / 2}}-\frac{3\left|\nabla H^2\right|^2}{4 c^4\left(1+H^2 / c^2\right)^{5 / 2}} \\
& \geq \frac{H^4}{n c^2\left(1+H^2 / c^2\right)^{3 / 2}}-\frac{3\left|\nabla H^2\right|^2}{4 c^4\left(1+H^2 / c^2\right)^{5 / 2}}
\end{aligned}
$$
which is equivalent to
$$
\frac{\Delta_{c h} f}{\sqrt{1+H^2 / c^2}}+\frac{3\left|\nabla H^2\right|^2}{4 c^4\left(1+H^2 / c^2\right)^3} \geq \frac{H^4}{n c^2\left(1+H^2 / c^2\right)^2} .
$$
It is not hard to see that the previous equation is equivalent to
$$
-f \Delta_{c h} f+3|\nabla f|^2 \geq \frac{c^2}{n}\left(1-f^2\right)^2 .
$$

By the Omori-Yau maximum principle for $\Delta_{c h}$, there exists a sequence $\left\{x_k\right\}_{k\in \n}$ in $\Sigma$ with the properties \eqref{eq:omori-yau}. Hence, taking limit as $k \to \infty$, we get $$\frac{c^2}{n} (1-(\sup f)^2)^2 \leq 0,$$
which means that $\sup f=-1.$ This implies that $H^2/c^2=0$, which is a contradiction because, for a translator, $H^2/c^2 \geq 1$ (recall \eqref{soliton-scalar}). 
\end{proof}

Now, our main results of nonexistence of solitons can be obtained.

\begin{theorem} \label{tII} 
Let $(M,g_M)$ be a complete Riemannian manifold with $K_M\geq 0$ and Ric$_M\geq -c, (c\geq 0)$. Let  $\mt=M \times \mathbb{R}$ endowed with the metric $\gt=g_M-d t^2$. 

Let  $\Sigma$ be a  spacelike translating soliton with respect $g=F^*(\gt)$ and let $r(\cdot)$ be the distance in $\Sigma$ to a fixed point $o\in \Sigma.$

Then, the mean curvature $H (>0)$ of $\Sigma$ cannot be bounded as
$$H(x) \leq G(r(x)), \qquad \forall x \in \Sigma \backslash B(o,R),$$
for any function 
 $G \in C^0([0,+\infty[)$ satisfying: 
 
 (a)  $G>0$, (b) $\int_0^\infty\frac{dt}{G(t)}=\infty$ and (c) $G$ is nondecreasing.

{ \noindent{In case $c=0$, the conditions (b) and (c) can be dropped.}}
\end{theorem} 

\begin{proof}
Reasoning by contradiction, notice that
the completeness of the soliton is ensured by Theorem \ref{t_completeness} and the applicability of the Omori-Yau principle to the drift Laplacian of $h$ by Theorem \ref{th:A}.
    Then, the contradiction follows from {Theorem}~\ref{th:nonexistence-1}.
    
{For the last assertion, apply Remark \ref{re:k=0}. }
\end{proof}

As a consequence, we can  reformulate this result from the viewpoint of the soliton graph equation, which essentially captures our result (recall Remark \ref{p_cover}).

\begin{proof}[Proof of Corollary I] 
    noticing that \eqref{e_soliton} is the equation of graph solitons coming from Example \ref{product-example}, $|\nabla^M u|<1$ is
    the restriction of being spacelike and $g$ is the metric on $M$ induced by the graph from the metric $\gt$ on $\mt$, {as well as the 
    criterion for $r_M$ bounds (Remark \ref{r_GM}) for the assertion on this distance.}
\end{proof}

\section{An illustrative example}  \label{sec:counter-example}
\vspace{.4cm}

In this concluding section, we aim to illustrate the construction of a {soliton with linear growth of $H$ respect to $r$.  
It also shows the sharpness of the bounds in terms of $r$ respect to the bounds respect to $g_M$}. 
Specifically, our example underscores the necessity of a lower bound in the sectional curvature to ensure the validity of Theorem II and Corollary I. 

Here, $M$ has dimension 2, thus, Ricci and sectional curvatures are equivalent (so, the non-negativeness of the former would imply the non-existence result by Remark \ref{re:k=0} and Theorem \ref{th:nonexistence-1}.)
In higher dimensions  we are imposing conditions on the sectional curvature, to ensure the Omori-Yau principle for the drift Laplacian (Theorem~\ref{th:A}), and on the Ricci one, to apply Theorem~\ref{th:nonexistence-1} using formula ~\eqref{f_deltaH}, consistently with \cite{BL}. The relation among these curvatures  would deserve a separated study, as they involve  deep properties of manifolds with non-negative Ricci curvature (see \cite{AG} and related references).

\subsection{Complete Riemannian surface $(M,g_M)$ with $K$  non lowerly bounded.} 
Consider $M=\R \times ]-1,1[$ and a smooth map $\phi:]-1,1[ \to ]0,2]$  such that:
\begin{enumerate}
\item[i)] $\phi(y)=1-|y|, \forall y \in \big]-1,-\frac{1}{2}\big[ \cup \big] \frac{1}{2}, 1 \big[$.
\item[ii)] $|\phi'(y)|<1, \forall y \in ]-1/2,1/2[$,
\item[iii)] $\phi''(y) \leq 0, \forall y \in ]-1/2,1/2[$.
\end{enumerate}
On $M$ we define the following Riemannian metric
$$ g_M=\frac{1}{\phi(y)^2}(dx^2+dy^2) \;.$$
The completeness of $g_M$ follows from $\int_{-1}^{-1/2}dy/\phi(y)=\int_{1/2}^1 dy/\phi(y)=\infty$.
 As $g_M={\rm e}^{2 \omega} g_0$ for $\omega=\ln\left(\frac{1}{\phi(y)} \right)$ and $g_0$ the Euclidean metric on $M$, its Gauss curvature is 
$$K_{g_M}(x,y)= {\rm e}^{2 \omega(y)}  (K_{g_0}(x,y)-\Delta_{g_0} \omega(y))= -{\rm e}^{2 \omega(y)}\left(\left(\frac{\phi'(y)}{\phi(y)}\right)^2-\frac{\phi''(y)}{\phi(y)} \right)<0.$$
In particular, $K_{g_M}$ is not lowerly bounded when 
$y\rightarrow \pm 1$.

\subsection{The soliton graph}
Consider on $M$ the function $F:M \to \R$ given by $F(x,y)=f(y)$, where  $f:]-1,1[ \to \R$ is smooth. $\Sigma:= {\rm graph}(F)$ is a spacelike translating soliton of velocity $c=1$, if and only if the function $f$ satisfies:

\begin{equation} \label{eq:tsgrafo} {\rm div}^M \left( \frac{\nablam f}{\sqrt{1-|\nablam f|^2}} \right)=
\frac{1}{\sqrt{1-|\nablam f|^2}} \, .
\end{equation}
Taking into account that $\nablam f=(0,\phi(y)^2 f'(y))$, the right-hand side of \eqref{eq:tsgrafo} is
$$ 
1/\sqrt{1-(\phi(y) f'(y))^2}\;  $$
and, using ${\rm div}^M (\sum_i^n  a^i\partial_i)=
\frac{1}{\sqrt{\det g}} \sum_{i=1}^n \frac{\partial}{\partial x^i}(a^i \sqrt{\det g})$, the left-hand one,

$$\phi(y)^2 \partial_y \left( \frac{f'(y)}{\sqrt{1-(\phi(y) f'(y))^2}} \right)= \phi(y)^2 \frac{\phi(y) \phi'(y)f'(y)^3+f''(y)}
{(1-(\phi(y) f'(y))^2)^{\frac{3}{2}}} \, .
$$ 
Then, \eqref{eq:tsgrafo} is equivalent to the ODE:
\begin{equation} \label{eq:edof} -1 + \phi(y)^2 \left( f'(y)^2(1+\phi(y)f'(y)\phi'(y))+f''(y)\right)=0\, .\end{equation}
It is  known that there exists a unique solution  $f:]-\varepsilon,\varepsilon[ \to \R$ satisfying $f(0)=f'(0)=0$ for some $\epsilon>0$, and let us  prove that its maximal domain  is $]-1,1[$. 

 \begin{figure}[htbp]
\begin{center}
\includegraphics[height=8.5cm]{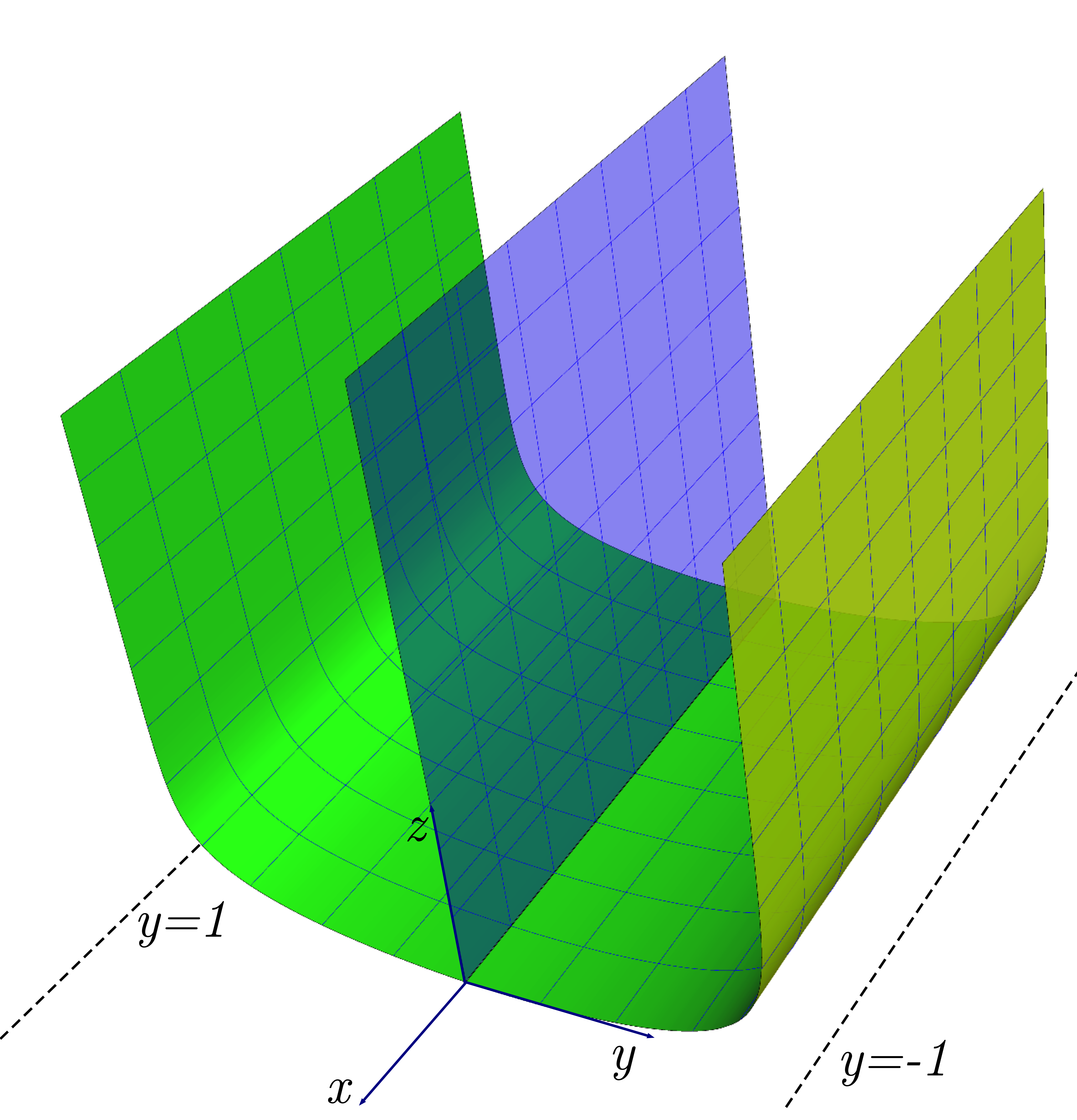}
\caption{\small The translator $\Sigma$ in $\mt$. The hypersurface $\{y=0\}$  appears in purple}
\label{fig:example                        g}
\end{center}
\end{figure}
Putting $z(y)=\phi(y) f'(y)$, formula \eqref{eq:edof} can be rewritten as
\begin{equation} \label{eq:edoz} 
 \phi(y) z'(y)=(1-z(y)^2)(\phi'(y) z(y)+1) . 
\end{equation}
Since we are assuming that $\Sigma$ is a spacelike surface and $\|\nablam f\|^2=(\phi^2 f')^2/\phi^2=z^2$, we must ensure $z^2<1$ everywhere. This happens around $0$, as $z(0)=0$.
By the hypotheses on $\phi$, we have   $\phi>0$ and $|\phi'|\leq 1$, so that \eqref{eq:edoz} yields 
\begin{equation} \label{eq:zz}
\mbox{$z'(y)>0$, whenever $|z(y)|<1$.} \end{equation}
Assume by contradiction that there exists a first point $y_+\in ]0,1[$ such that  $z(y_+)=1$ (an analogous contradiction would follow with $y_-\in ]-1,0[$). Integrating in \eqref{eq:edoz}

\begin{equation} \label{eq:edovs} 
\int_{0}^{y_+ (=y(1))}\frac{dy}{\phi(y)}
 = \int_0^{1(=z(y_+))} \frac{dz}{(1-z^2)(\phi'(y(z)) z+1)} 
\geq \frac{1}{2} \int_0^{1} \frac{dz}{1-z^2}=\infty,
\end{equation}
(recall $|\phi'|\leq 1$) which is a contradiction, because $\phi|_{[0,y_+]}$  is bounded away from 0. 

Notice that the integral expression  permits to obtain $y\mapsto z(y)$ and, then $f(y)$, thus defining the soliton on the whole interval $]-1,1[$.
What is more, the condition i) on $\phi$ implies that, replacing the limits of integration $0, y_+$ by, respectively, $1/2, 1$
(and analogously in the direction towards $-1$):  
$$\lim_{y\to -1} z(y)=-1, \quad \lim_{y\to 1} z(y)=1.$$ Summarizing, one has

\begin{proposition}
    \label{claim:z}
    On $(M,g_M)$, the graph soliton  $F(x,y)=f(y)$ with $f(0)=f'(0)=0$ is determined by $z:=\phi \cdot f':(-1,1) \rightarrow (-1,1)$. This function  is strictly increasing with limits
    $$\lim_{y\to -1^+} z(y)=-1, \quad \lim_{y\to 1^-} z(y)=1.$$
\end{proposition}

\subsection{Expression of the intrinsic metric of the soliton and completeness} 
As $\Sigma$ was the graph of $F(x,y)=f(y)$, the intrinsic metric $g$ on $\Sigma$ is isometric to $F^*\gt$ on $M$ and we can put
$$
g=g_M-dF^2=\frac{1}{\phi(y)^2}\left(dx^2+(1-z^2(y))dy^2\right), \qquad \forall (x,y)\in M.
$$
From Lemma \ref{claim:z}, $z$ can replace $y$ as a coordinate on $M$, and we can introduce a new one $w=w(y(z))$ as follows:
$$
dw= \frac{\sqrt{1-z^2(y)}}{\phi(y)}dy=
\frac{dz}{\sqrt{(1-z^2}\; (\phi'(y(z)) z+1)},
$$
the latter taking into account \eqref{eq:edovs}. 
Thus, taking  $w(z)$ as $\int_0^z dw$ and using
 $\phi' (y(z))=-1$ (resp. $=1$) in an interval   $]z_+, 1[$
(resp. $]-1.z_-[$),  
\begin{eqnarray} 
    w(z) & = & \frac{\sqrt{1+z}}{\sqrt{1-z}}+c_+ \quad z_+<z<1. \label{w1} \\
w(z) & = & -\frac{\sqrt{1-z}}{\sqrt{1+z}}+c_- \quad -1<z<z_- \label{w2}
\end{eqnarray}
where $z_+:=z(y=1/2) \in ]0,1[$ and $z_-:=z(y=-1/2) \in ]-1,0[$.
In particular, 
$\lim_{z\rightarrow {\pm 1}} w(z)= \pm \infty$ and
we can write:
\begin{equation}
    \label{eq:metric}
    g= \frac{1}{\phi(y(z(w)))^2} dx^2+dw^2, \qquad \forall (x,w)\in \R^2.
\end{equation}

\begin{remark} As $H$ will be affinely bounded, the completeness of $g$ will be a consequence of the bound of $H$ obtained below, applying Corollary \ref{co:completeness} (1b). 
However,  completeness can also be checked directly, because $g$ is the warped product $B\times _\psi F$
of two complete Riemannian manifolds (with base 
$(B,g_B)$ and the fiber $(F,g_F)$  equal to $\R$, and warping function $\psi(w)=1/\phi(y(z(w))$) and, therefore, its completeness is well-known
(see for example \cite[Lemma 7.40]{Oneill}). 

\end{remark}

\subsection{Growth of $H$ in terms of $r$.} As $\nu=(f' \, \phi \partial_y-\partial_t)/\sqrt{1-z^2}$, from \eqref{soliton-scalar},
\begin{equation} \label{eq:HH} H=\frac{1}{\sqrt{1-z^2}}.\end{equation}
Fix  $o=(0,0)$, define (as before) $r(x,w):= d(o,(x,w))$ and notice from \eqref{eq:metric}:
\begin{equation}\label{e_rxw}
    r(x,w) \geq r(0,w) = |w|. 
\end{equation}
Using this inequality and, then,  \eqref{w1}, \eqref{w2} as well that $H(x,w)$ is independent of $x$, we have that
\begin{equation}\label{e_rxw2}
    \limsup_{r(x,w) \to +\infty} \frac{H}{r}(x,w)  \leq \limsup_{w \to +\infty} \frac{H(0,w)}{w} =\frac 12.
\end{equation}
In particular, it is possible to find  constants $A,B >0$ so that:
$$H(p) \leq A \; r(p)+B \quad \forall p \in \Sigma,$$
that is, the growth is at most linear with 
 $G(r) =  A \, r+B$. 
 
 Remarkably,  $A\geq  1/2$ must hold, 
   because the equality in \eqref{e_rxw} (and, then, in \eqref{e_rxw2}) is attained when $x=0$. 
 \subsection{Growth of $H$ in terms of $r_M$.} 

 Let us compute $H$ in terms of $r_M$ close to $y=1$. We have, from \eqref{eq:edovs}, that
 $$\frac{dy}{1-y}=\frac{dz}{(1-z^2)(1-z)}\geq C_1 \frac{dz}{(1-z)^2}, \quad C_1>0.$$
 Then, {there exists $C_2\in \R$ such that} 
 $$-\ln(1-y) \geq \frac{C_1}{1-z}-C_2 . 
 $$
 On the other hand, taking $o_M=(0,0)$, then $r_M(0,y)\equiv r_M(y)=-\ln (1-y).$ Thus
 \begin{equation} 
     z \leq 1-\frac{C_1}{r_M+C_2}
 \end{equation}
{focusing in points with $r_M(y)$ large}. Using the above inequality in \eqref{eq:HH}, 
 $$H(0,y) \leq C_3 \sqrt{r_M(0,y)+C_2}, 
 \qquad  \mbox{for some $C_3>0$.}$$
{Choosing suitable constants, $H$ could also be bounded analogously from below.  That is, the logarithmic $g_M$-bound required to ensure the Omori-Yau maximum principle (stressed in Remark \ref{r_GM}) cannot be obtained, underlining the sharpness of the estimates in terms of $r$}. 
 
 In conclusion, the case of unbounded $H$ with controlled growth studied here seems an especially relevant issue for the development of the theory of solitons in the Lorentzian setting.
 

\end{document}